\documentclass[9.5pt, a4paper]{article}
\usepackage{amsmath}
\usepackage{fancyhdr}
\usepackage[misc]{ifsym}
\usepackage{color}
\usepackage{amsmath}
\usepackage{amsfonts}
\usepackage{dsfont}
\usepackage{mathrsfs}
\usepackage{mathbbol}
\usepackage{bbm}
\usepackage{latexsym}
\usepackage{graphicx}
\usepackage{epstopdf}
\usepackage{amssymb}
\usepackage{indentfirst}
\usepackage{subfigure}
\usepackage{booktabs}
\usepackage{authblk}
\usepackage{cite}
\usepackage{float}
\usepackage[numbers,sort&compress]{natbib}
\usepackage{caption}
\usepackage{booktabs}
\usepackage{threeparttable}
\usepackage{multicol}
\usepackage{multirow}
\usepackage{geometry}
\usepackage{lineno}
\usepackage{amsbsy}
\usepackage{bm}
\usepackage{dcolumn}
\usepackage{mathrsfs}

\newtheorem{remark}{Remark}[section]
\newtheorem{theorem}{Theorem}[section]
\newtheorem{lemma}[theorem]{Lemma}
\newtheorem{corollary}[theorem]{Corollary}

\newtheorem{example}[theorem]{Example}

\newtheorem{assumption}{Assumption}
\numberwithin{equation}{section}

\def\qed{\hfill$\Box$\vspace{8pt}}

\begin{document}
\title{More than one Author with different Affiliations}
\author[1]{Qigang Liang}
\author[1,2]{Xuejun Xu}

\affil[1]{\small School of Mathematical Science, Tongji University, Shanghai 200092, China, qigang$\_$liang@tongji.edu.cn}
\affil[2]{\small Institute of Computational Mathematics, AMSS, Chinese Academy of Sciences, Beijing 100190, China, xxj@lsec.cc.ac.cn}
\title{A Two-Level Preconditioned Helmholtz-Jacobi-Davidson Method for the Maxwell Eigenvalue Problem}\date{}
\maketitle

{\bf{Abstract}:}\ \ In this paper, based on a domain decomposition (DD) method, we shall propose an efficient two-level preconditioned Helmholtz-Jacobi-Davidson (PHJD) method for solving the algebraic eigenvalue problem resulting from the edge element approximation of the Maxwell eigenvalue problem. In order to eliminate the components in orthogonal complement space of the eigenvalue, we shall solve a parallel preconditioned system and a Helmholtz projection system together in fine space. After one coarse space correction in each iteration and minimizing the Rayleigh quotient in a small dimensional Davidson space, we finally get the error reduction of this two-level PHJD method as $\gamma=c(H)(1-C\frac{\delta^{2}}{H^{2}})$, where $C$ is a constant independent of the mesh size $h$ and the diameter of subdomains $H$, $\delta$ is the overlapping size among the subdomains, and $c(H)$ decreasing as $H\to 0$, which means the greater the number of subdomains, the better the convergence rate. Numerical results supporting our theory shall be given.\\

{\bf{Keywords}:}\ \ Maxwell eigenvalue problem, edge element, Helmholtz projection, Jacobi-Davidson method, domain decomposition. \hspace*{2pt}

\section{Introduction}
In this paper, we develop an efficient numerical algorithm for solving the Maxwell eigenvalue problem which plays an important role in computational electromagnetism (see, e.g., \cite{Hiptmair1,Boffi,Liu,Jiang1,Jiang2,CA}). The governing equations are
\begin{equation}
    \begin{cases}\label{MaxwellEquationsPri}
     \bm{curl}\bm{E}=-j\omega\mu\bm{H}  &\text{$ x \in \Omega $},\\
     div(\epsilon \bm{E})=0    &\text {$ x \in \Omega $},\\
     \bm{curl}\bm{H}=j\omega\epsilon \bm{E}   &\text {$ x \in \Omega $},\\
     div(\mu\bm{H})=0         &\text {$ x \in \Omega $},\\
     \bm{n}\times \bm{E}=\bm{0}    &\text {$ x \in \partial \Omega $},\\
    (\mu\bm{H})\cdot\bm{n}=0  &\text {$ x \in \partial \Omega $},\\
    \end{cases}
\end{equation}
and we eliminate the magnetic field $\bm{H}$, then obtain the equivalent eigenvalue problem for the electric field $\bm{E}$ as follows:
\begin{equation}
     \begin{cases}\label{MaxwellEigenvalue}
          \bm{curl}\mu^{-1}\bm{curl}\bm{E}=\omega^{2}\epsilon\bm{E}\ \ \ &\text{$x\in \Omega$},\\
          div(\epsilon\bm{E})=0\ \ \  &\text{$x\in \Omega$},\\
          \bm{n}\times \bm{E}=\bm{0}\ \ \ &\text{$x\in \partial\Omega$},
     \end{cases}
\end{equation}
where the resonant cavity $\Omega\subset R^{3}$ is  a bounded convex polyhedral domain, the coefficients $\epsilon\ $and $\mu\ $ are the real electric permittivity and magnetic permeability, respectively, $j$ is an imaginary unit, and the notation $\omega$ is the resonant angular frequency of the electromagnetic wave for cavity $\Omega$.\ For convenience, we denote that
\[\bm{u}:=\bm{E},\ \ \ \ \ \lambda:=\epsilon_{0}\mu_{0}\omega^{2},\]
where $\epsilon_{0}$ and $\mu_{0}$ are electric permittivity and magnetic permeability in vacuum, respectively.
So equations $\eqref{MaxwellEigenvalue}$ may be rewritten as:
\begin{equation}
     \begin{cases}\label{MaxwellEigenvalueu}
          \bm{curl}\mu_{r}^{-1}\bm{curl}\bm{u}=\lambda \epsilon_{r}\bm{u}\ \ \ &\text{$x\in \Omega$},\\
          div(\epsilon_{r}\bm{u})=0\ \ \  &\text{$x\in \Omega$},\\
          \bm{n}\times \bm{u}=\bm{0}\ \ \ &\text{$x\in \partial\Omega$},
     \end{cases}
\end{equation}
where the coefficients $\epsilon_{r}\ $and $\mu_{r}\ $ are the real relative electric permittivity and magnetic permeability, respectively. For simplicity, we consider that the media is isotropic and homogeneous, i.e., the coefficients $\epsilon_{r}(\geq 1)$ and $\mu_{r}(\geq 1)$ both are positive constants.
\par Edge elements were introduced by N$\acute{e}$d$\acute{e}$lec (see, e.g., \cite{ND1,ND2}). Actually, the lowest-order edge elements are often referred to as the Whitney elements (see, e.g., \cite{Hiptmair1,Boffi} and therein references). As edge elements may eliminate the nonphysical spurious nonzero eigenvalues, they have been widely used and studied to solve the Maxwell eigenvalue problem. It is known that the challenging task for solving the Maxwell eigenvalue problem is how to deal with the infinite dimensional kernel of the $\bm{curl}$ operator. Owing to this difficulty, one of approaches is chosen to so-called the penalty method, which imposes divergence-free condition by introducing the penalty term (see \cite{Buffa}), but it usually results in spurious eigenvalues. Alternative approach is the mixed variational method, which is introduced by Kikuchi (see \cite{Kikuchi}) and Boffi (see, e.g., \cite{Boffi1,Boffi}). The method may be used to handle the kernel of the $\bm{curl}$ operator but simultaneously brings larger scale size computational problem and difficult saddle-problem. The standard method drops the divergence-free condition, though it may induce spurious zero frequency, doing so does not contaminate the nonzero eigenvalues (see, e.g., \cite{Hiptmair1,Boffi}). Two grid methods have been widely used to solve elliptic eigenvalue problems (see, e.g., \cite{Aihui,Xiaoliang,Yang1}). Its application to the Maxwell eigenvalue problem has been considered in \cite{chenlong}. The idea of two grid methods is that one may first compute an eigenvalue problem on the coarse grid and then compute a boundary value problem on the fine grid. Rayleigh quotient may be used to accelerate the approximation of the eigenvalue. Under the assumptions $h=O(H^{i}) (i=2,3,4$, see, \cite{Aihui}, \cite{chenlong}, \cite{Yang1}), asymptotically optimal error estimates may be obtained, respectively. Based on the de Rham complex, Hiptmair and Neymeyr (see \cite{Hiptmair2}) proposed a projected preconditioned inverse iterative method. It can handle the kernel of the $\bm{curl}$ operator by the Helmholtz projection. A multilevel preconditioned technique is chosen to accelerate the inverse iteration procedure. This idea is also extended to the adaptive discretization for the Maxwell eigenvalue problem (see \cite{Yifeng}).
\par  The two-level preconditioned Jacobi-Davidson methods for eigenvalue problems were proposed by Zhao, Hwang and Cai (see \cite{Cai}) and further developed and analyzed by Wang and Xu (see \cite{WX1,WX2}) for $2m$-order $(m=1, 2)$ elliptic eigenvalue problems. The convergence rate of the first eigenvalue of the elliptic eigenvalue problems is bounded by $c(H)(1-C\frac{\delta^{2m-1}}{H^{2m-1}})^{2}$ $(m=1,2)$. The first eigenvalue of the partial differential operator is usually so-called principal eigenvalue(see, e.g., \cite{Evans}). In this paper, based a domain decomposition method, we shall propose an efficient and highly parallel two-level preconditioned Helmholtz-Jacobi-Davidson method (PHJD) for solving the Maxwell eigenvalue problem. We shall prove that the convergence rate of the principal eigenvalue is bounded by $c(H)(1-C\frac{\delta^{2}}{H^{2}})^{2}$, i.e.,
\begin{equation}\label{introeigenvaluerate}
\lambda^{k+1}-\lambda_{1}^{h}\leq c(H)(1-C\frac{\delta^{2}}{H^{2}})^{2}(\lambda^{k}-\lambda_{1}^{h}),
\end{equation}
where $C$ is a constant independent of the mesh size $h$ and the diameter of subdomains $H$, $\delta$ is the overlapping size among the subdomains, $\lambda_{1}^{h}$ is the first discretized eigenvalue for the Maxwell eigenvalue problem and $\lambda^{k}$ is the $k$-th iteration of the two-level PHJD method. In addition, our algorithm works very well when $h<<H^{4}$ in contrast to the two grid method which needs $h=O(H^{3})$ for the Maxwell eigenvalue problem(see \cite{chenlong}). Meanwhile, our PHJD method holds good scalabilities, which shall be verified by our numerical experiments.
\par We must emphasize that the application of the two-level domain decomposition preconditioned algorithm to the Maxwell eigenvalue problem is nontrivial. The first difficulty is that the condition
\begin{equation}\label{introcondi1}
|\lambda_{1}^{h}-\lambda_{1}|+||u-u_{h}||_{0}+h||u-u_{h}||_{a}=O(h^{2})
\end{equation}
holds for $2m$-order($m=1,2$) elliptic problems (see \cite{WX2}), but it is not true for the first-type edge elements because the polynomial space is incomplete; The second difficulty is that the kernel of the $\bm{curl}$ operator in $\bm{H}_{0}(\bm{curl};\Omega)$, which is $\nabla{H_{0}^{1}(\Omega)}$ in trivial topology domain, is very large while the kernel of the gradient operator in $H_{0}^{1}(\Omega)$ is zero space. Hence, if we take advantage of Jacobi-Davidson method to solve the algebraic system resulting from the edge element approximation of the Maxwell eigenvalue problem directly, we shall fail to work due to the fact that the iterative solution may plunge into the kernel of the $\bm{curl}$ operator; The third difficulty is that the discrete divergence-free space is non-nested for $\bm{H}_{0}(\bm{curl};\Omega)\cap \bm{H}(div_{0};\Omega;\epsilon_{r})$ (defined in the following, $\eqref{curldefinition},\eqref{div0definition}$), which means that the discrete edge element space is essentially nonconforming for the space $\bm{H}_{0}(\bm{curl};\Omega)\cap \bm{H}(div_{0};\Omega;\epsilon_{r})$. In this paper, we shall overcome these difficulties, and try to prove the convergence result of the principal eigenvalue is almost near optimal, i.e., $\eqref{introeigenvaluerate}$ is true.

\par The outline of this paper is organized as follows: Some preliminaries are introduced in Section 2. In Section 3, our two-level  preconditioned Helmholtz-Jacobi-Davidson method for the Maxwell eigenvalue problem is proposed. Some useful lemmas and the main convergence analysis are given in Section 4. Finally we present our numerical results in Section 5.

\section{The model problem and preliminary}

Let $\Omega\subset R^d,\ d=2,3$ be a bounded convex polygonal or polyhedral domain. We use the standard notations for the Sobolev spaces and denote
\begin{equation}\label{curldefinition}
\bm{H}_{0}(\bm{curl};\Omega):=
\begin{cases}
\{\bm{u}\in L^{2}(\Omega)^{2}\ |\ curl\bm{u}\in L^{2}(\Omega),\bm{u}\cdot\bm{t} |_{\partial{\Omega}}=0\ \}  & (if\  d=2),\\
\{\bm{u}\in L^{2}(\Omega)^{3}\ |\ \bm{curl}\bm{u}\in L^{2}(\Omega)^{3},\bm{u}\times \bm{n}|_{\partial{\Omega}}=\bm{0}\ \}  & (if\  d=3),\\
\end{cases}
\end{equation}
equipped with the norm $||\bm{u}||_{\bm{curl}}:=\{||\bm{u}||^{2}+||\bm{curl}\bm{u}||^{2}\}^{\frac{1}{2}}$, where $curl\bm{u}\ (d=2)$ denotes a
scale variable $\frac{\partial{u_{2}}}{\partial{x_{1}}}-\frac{\partial{u_{1}}}{\partial{x_{2}}}$($\frac{\partial}{\partial{x_{i}}}$ means a weak derivative),
$\bm{curl}\bm{u}\ (d=3)$ denotes a vector $(\frac{\partial{u_{3}}}{\partial{x_{2}}}-\frac{\partial{u_{2}}}{\partial{x_{3}}},
\frac{\partial{u_{1}}}{\partial{x_{3}}}-\frac{\partial{u_{3}}}{\partial{x_{1}}},
\frac{\partial{u_{2}}}{\partial{x_{1}}}-\frac{\partial{u_{1}}}{\partial{x_{2}}})$,
$\bm{n}$ denotes the outer unit normal vector, $\bm{t}$ denotes the unit tangential vector along the boundary $\partial{\Omega}$, and $||\cdot||$ symbols usual $L^{2}$-norm induced by $L^{2}$-inner product $(\cdot,\cdot)$. We also denote that
\begin{equation}\notag
\bm{H}(div;\Omega;\epsilon_{r}):=\{\bm{u}\in L^{2}(\Omega)^{d}\ |\ div(\epsilon_{r}\bm{u})\in L^{2}(\Omega)\  \}\ \ \ \text{$(d=2,3)$}
\end{equation}
and
\begin{equation}\label{div0definition}
\bm{H}(div_{0};\Omega;\epsilon_{r}):=\{\bm{u}\in \bm{H}(div;\Omega;\epsilon_{r}) \ |\ div(\epsilon_{r}\bm{u})=0 \  \}\ \ \ \text{$(d=2,3)$},
\end{equation}
equipped with the norm $||\bm{u}||_{div}:=\{||\bm{u}||^{2}+||div(\epsilon_{r}\bm{u})||^{2}\}^{\frac{1}{2}}$, where $div\bm{u}$ denotes $\sum_{i=1}^{d}\frac{\partial{u_{i}}}{\partial{x_{i}}}$.
For the convenience of symbols, we only consider 3-dimensional Maxwell eigenvalue problem in the following, but our algorithm also works very well and theoretical results also hold for 2-dimensional Maxwell eigenvalue problem.

\subsection{Maxwell eigenvalue problem}
\par We focus on the governing equations
$\eqref{MaxwellEigenvalueu}$. It is known that the equivalent variational form is as follows(see, e.g., \cite{Boffi}):
\begin{equation}\label{continuosvariation}
             \begin{cases}
             \text{Find $(\lambda,\bm{u})\in R\times \bm{H}_{0}(\bm{curl};\Omega)$, such that $\lambda>0,||\bm{u}||_{b}=1$},\\
             a(\bm{u},\bm{v})=\lambda b(\bm{u},\bm{v})\ \ \ \ \forall\ \bm{v}\in \bm{H}_{0}(\bm{curl};\Omega),
             \end{cases}
\end{equation}
where $a(\bm{u},\bm{v}):=(\mu_{r}^{-1}\bm{curl}\bm{u},\bm{curl}\bm{v}),\ b(\bm{u},\bm{v}):=(\epsilon_{r}\bm{u},\bm{v})$. Define the norm $||\cdot||_{b}:=\sqrt{b(\cdot,\cdot)}$ in $\bm{H}_{0}(\bm{curl};\Omega)$. Due to the fact that the Poincar$\acute{e}$ inequality holds in $\bm{H}_{0}(\bm{curl};\Omega)\cap \bm{H}(div_{0};\Omega;\epsilon_{r})$, we may define the norm
$||\cdot||_{a}:=\sqrt{a(\cdot,\cdot)}$ in $\bm{H}_{0}(\bm{curl};\Omega)\cap \bm{H}(div_{0};\Omega;\epsilon_{r})$. For any $\bm{f}\in L^{2}(\Omega)^{3}$, define $T: L^{2}(\Omega)^{3}\to \bm{H}_{0}(\bm{curl};\Omega)\cap \bm{H}(div_{0};\Omega;\epsilon_{r})$, such that
\begin{equation}\label{div0}
a(T\bm{f},\bm{v})=b(\bm{f},\bm{v})\ \ \ \forall\ \bm{v}\in \bm{H}_{0}(\bm{curl};\Omega)\cap \bm{H}(div_{0};\Omega;\epsilon_{r}).
\end{equation}
Combining with the Poincar$\acute{e}$ inequality, we know that the definition of the operator $T$ is meaningful. It is obvious that $T$ is symmetric in the sense of $b(\cdot,\cdot)$ because of the definitions of $a(\cdot,\cdot)$ and $b(\cdot,\cdot)$. Due to the fact that $\bm{H}_{0}(\bm{curl};\Omega)\cap \bm{H}(div_{0};\Omega;\epsilon_{r})$ is embedded compactly in $L^{2}(\Omega)^{3}$ (see, e.g., \cite{Amr,Hiptmair1}), we know that the operator $T:L^{2}(\Omega)^{3}\to L^{2}(\Omega)^{3}$ is a compact operator. Furthermore, the operator $T:(L^{2}(\Omega)^{3},b(\cdot,\cdot))\to (L^{2}(\Omega)^{3},b(\cdot,\cdot))$ is also a compact operator. Hence, owing to the well-known Riesz-Schauder theorem, it is known that
\begin{equation}\notag
T\bm{u}_{i}=\frac{1}{\lambda_{i}}\bm{u}_{i}.
\end{equation}
Meanwhile,
\begin{equation}\notag
0<\lambda_{1}\leq \lambda_{2}\leq \lambda_{3}\leq...\leq\lambda_{n}\rightarrow +\infty\ \ \ \ n\to +\infty,
\end{equation}
and corresponding eigenvectors are
\begin{equation}\notag
\bm{u}_{1},\bm{u}_{2},\bm{u}_{3},...,
\end{equation}
which satisfy
\begin{equation}\notag
a(\bm{u}_{i},\bm{u}_{j})=\lambda_{i}b(\bm{u}_{i},\bm{u}_{j})=\delta_{ij}\lambda_{i},
\end{equation}
where $\delta_{ij}$ is the Kronecker notation. Obviously the eigen-pair in $\eqref{continuosvariation}$ and the eigen-pair of the operator $T$ are all  corresponding(see \cite{Boffi1,chenlong}).

\subsection{ Finite element discretization}
\par  We use the standard edge elements to discretize $\bm{H}_{0}(\bm{curl};\Omega)$ and our ultimate conclusions are true for all other edge elements(see \cite{ND1,ND2}). For simplicity, we only consider the lowest-order edge element. The local polynomial space is
\begin{equation}\label{ND}
ND(K)=\{\bm{v}\ |\ \bm{v}=\bm{a}+\bm{b}\times \bm{x},\ \ \ \bm{a},\bm{b}\in \bm{R}^{3}\ \bm{x}\in K\  \},
\end{equation}
and the moments are
\begin{equation}\notag
M(\bm{v},p_{0},e_{ij})=\int_{e_{ij}}\bm{v}\cdot\bm{t}_{ij}p_{0}ds\ \ \forall\  p_{0}\in P_{0}(K), \ \ 1\leq i<j\leq 4,
\end{equation}
where $\bm{t}_{ij}$ is the unit tangential vector along edge $e_{ij}$, and local basic vector fields are
\begin{equation}\notag
\bm{\varphi}^{K}_{ij}=\lambda_{i}^{K}\nabla{\lambda_{j}^{K}}-\lambda_{j}^{K}\nabla{\lambda_{i}^{K}},\ \ 1\leq i<j\leq 4,
\end{equation}
where the $\lambda_{i}^{K}$ is the barycentric coordinate corresponding to the $i$-th node on the tetrahedron $K$.
So our edge element space is
\begin{equation}\notag
E_{h}(\Omega_{h}):=\{\bm{v}\in \bm{H}_{0}(\bm{curl};\Omega)\ |\ \bm{v}|_{K}\in ND(K),\ \forall\ K \in \tau_{h}\ \}.
\end{equation}
Then the discrete standard variational form of $\eqref{continuosvariation}$ may be written as follows:
\begin{equation}
             \begin{cases}\label{discreteH0curl}
             \text{Find $(\lambda^{h},\bm{u}^{h})\in R\times E_{h}(\Omega_{h})$, such that $\lambda^{h}>0,||\bm{u}^{h}||_{b}=1$}\\
             a(\bm{u}^{h},\bm{v})=\lambda^{h} b(\bm{u}^{h},\bm{v})\ \ \ \ \forall\ \bm{v}\in E_{h}(\Omega_{h}).
             \end{cases}
\end{equation}
Define the discrete divergence-free space to be  $E_{h}^{0}(\Omega_{h};\epsilon_{r}):=\{\bm{v}\in E_{h}(\Omega_{h})\ |\ b(\bm{v},\nabla{p}_{h})=0,\ \ \forall\ p_{h}\in S_{h}(\Omega_{h})\ \}$, with the $S_{h}(\Omega_{h})$ denoting a continuous piecewise linear polynomial space in $\Omega$ with vanishing trace on $\partial{\Omega}$ and define the discrete operator $T_{h}:L^{2}(\Omega)^{3}\to E_{h}^{0}(\Omega_{h};\epsilon_{r})$ as follows: $\forall \bm{f}\in L^{2}(\Omega)^{3}$,
\begin{equation}\label{discretediv0}
a(T_{h}\bm{f},\bm{v})=b(\bm{f},\bm{v})\ \ \ \ \ \forall\ \bm{v}\in E_{h}^{0}(\Omega_{h};\epsilon_{r}).
\end{equation}
Due to the discrete Poincar$\acute{e}$ inequality (see, e.g., \cite{Hiptmair1}), we know that the definition of $T_{h}$ is meaningful and may define the norm $||\cdot||_{a}:=\sqrt{a(\cdot,\cdot)}$ in $E_{h}^{0}(\Omega_{h};\epsilon_{r})$. It is easy to see that $T_{h}$ is symmetric and compact. 
Because of the well-known Riesz-Schauder theorem, we have
\begin{equation}\notag
T_{h}\bm{u}_{i}^{h}=\frac{1}{\lambda_{i}^{h}}\bm{u}^{h}_{i}.
\end{equation}
Meanwhile,
\begin{equation}\notag
0<\lambda_{1}^{h}\leq \lambda_{2}^{h}\leq \lambda_{3}^{h}\leq...\leq \lambda_{nd}^{h},
\end{equation}
and the corresponding eigenvectors are
\begin{equation}\notag
\bm{u}_{1}^{h},\bm{u}_{2}^{h},\bm{u}_{3}^{h},...,\bm{u}_{nd}^{h},
\end{equation}
which satisfy
\begin{equation}\notag
a(\bm{u}^{h}_{i},\bm{u}^{h}_{j})=\lambda_{i}^{h}b(\bm{u}^{h}_{i},\bm{u}^{h}_{j})=\delta_{ij}\lambda^{h}_{i},
\end{equation}
where $nd=dim( E_{h}^{0}(\Omega_{h};\epsilon_{r}))$. Similar to the continuous case, the eigen-pair in $\eqref{discreteH0curl}$ and the eigen-pair of the operator $T_{h}$ are all corresponding. For the convenience of the following error estimate, we define the operator
$A^{h}:E_{h}(\Omega_{h})\to E_{h}(\Omega_{h})$ such that
\begin{equation}\notag
b(A^{h}\bm{v}^{h},\bm{w}^{h})=a(\bm{v}^{h},\bm{w}^{h})\ \ \ \ \forall\  \bm{w}^{h}\in E_{h}(\Omega_{h}),
\end{equation}
which is the discrete analog of the operator $ \bm{curl}\mu_{r}^{-1}\bm{curl} $ in $E_{h}(\Omega_{h})$.
\begin{remark}
In this paper, we are interested in the case that the principal eigenvalue is simple, i.e.,
\begin{equation}\notag
0<\lambda_{1}<\lambda_{2}\leq \lambda_{3}\leq...\leq \lambda_{n}\to +\infty, \ \ \ \ n \to +\infty,
\end{equation}
and for the corresponding discrete version, we also have
\begin{equation}\notag
0<\lambda_{1}^{h}<\lambda_{2}^{h}\leq \lambda_{3}^{h}\leq...\leq \lambda_{nd}^{h}.
\end{equation}
The principal eigenvalue with simple algebraic multiplicity is common. For example, the resonant domain $\Omega$ is a cuboid in $R^{3}$, where the length of each edge is different. Specifically, when the resonant cavity $\Omega=(0,L)\times(0,K)\times(0,R)$ is equipped with vacuum in $\eqref{MaxwellEigenvalueu}$, we may know that $\lambda=(\frac{\pi}{L})^{2}m^{2}+(\frac{\pi}{K})^{2}n^{2}+(\frac{\pi}{R})^{2}l^{2}$, where $m, n$ and $ l$ are nonnegative integers and at most one of elements in $\{m, n, l\}$ takes zero(see \cite{CA}).
\end{remark}

\par It is known that we have the following spacial decomposition properties(see \cite{Boffi})
\begin{equation}\label{continuousspectralspaces}
\bm{H}_{0}(\bm{curl};\Omega)=\nabla{H_{0}^{1}(\Omega)}\oplus M(\lambda_{1})\oplus M^{\perp}(\lambda_{1}),
\end{equation}
\begin{equation}\label{discretespectralspaces}
E_{h}(\Omega_{h})=\nabla{S_{h}(\Omega_{h})}\oplus M_{h}(\lambda_{1})\oplus M_{h}^{\perp}(\lambda_{1}),
\end{equation}
where $M(\lambda_{1})=span\{\bm{u}_{1}\},\ M_{h}(\lambda_{1})=span\{\bm{u}^{h}_{1}\}$ and the notation $\oplus$ denotes the $b(\cdot,\cdot)$-orthogonal(also $a(\cdot,\cdot)$-orthogonal) direct sum. For the convenience of the following symbols, we denote that $K_{0}^{h}:E_{h}(\Omega_{h})\to \nabla{S_{h}(\Omega_{h})},\ Q_{1}^{h}:E_{h}(\Omega_{h})\to M_{h}(\lambda_{1}),\ Q_{2}^{h}:E_{h}(\Omega_{h})\to M_{h}^{\perp}(\lambda_{1})$ are the $b(\cdot,\cdot)$-orthogonal projections. Similarly, we may define the operator $K_{0}^{H},Q_{1}^{H},Q_{2}^{H}$ on the coarse level and let
$Z^{(0)}:=I-K_{0}^{H}$.
\par It is obvious that $E_{h}^{0}(\Omega_{h};\epsilon_{r})\not\subset M(\lambda_{1})\oplus M^{\perp}(\lambda_{1})$. Denote $H_{\epsilon_{r}}: \bm{H}_{0}(\bm{curl};\Omega)\to M(\lambda_{1})\oplus M^{\perp}(\lambda_{1})$ to be $b(\cdot,\cdot)$-orthogonal projection, which is usually so-called Hodge operator, and let $V^{+}:=H_{\epsilon_{r}}E_{h}^{0}(\Omega_{h};\epsilon_{r})$. It is easy to see that the elements in $V^{+}$ are not finite element functions, but $\bm{curl}V^{+}\subset RT_{0}(\Omega_{h})$, where $RT_{0}(\Omega_{h})$ is a well-known Raviart-Thomas element space with vanishing normal trace along the boundary $\partial{\Omega}$. We define an operator $P^{h}:M(\lambda_{1})\oplus M^{\perp}(\lambda_{1})\to V^{+}$ as follows:
\begin{equation}\label{definitionofPh}
a(\bm{u}-P^{h}\bm{u},\bm{v})=0 \ \ \ \ \forall\ \bm{v}\in V^{+}.
\end{equation}
Furthermore, we may also define $P^{h}:\bm{H}_{0}(\bm{curl};\Omega)\to V^{+}$ by a trivial extension.
\par According to the definition of $P^{h}$, we may obtain the following lemma (see \cite{Toselli}).
\begin{lemma}\label{lemmaPh}
Let $\Omega$ be a convex bounded polyhedral domain, we have
\begin{equation}
||\bm{u}^{h}-P^{h}\bm{u}^{h}||_{b}\leq Ch||\bm{curl}\bm{u}^{h}||_{b}\ \ \ \text{$\forall\ \bm{u}^{h}\in E^{0}_{h}(\Omega_{h};\epsilon_{r})$}.
\end{equation}
\end{lemma}
\par The following a prior error estimates (see Theorem 5.4 in \cite{Boffi1}, Subsection 4.3 and Subsection 4.4 in \cite{Hiptmair1}) are useful in our convergence analysis. The proof of the estimates essentially needs strengthened discrete compactness properties and standard approximation properties. For interested readers, please refer to  Theorem 19.6 in \cite{Boffi} or \cite{Osb} for details.
\begin{theorem}\label{prioreTheorem}
Let $\Omega$ be a bounded convex polyhedral domain. There exists $h_{0}>0$ such that for any $h$ $(0<h<h_{0})$ and for any $\lambda_{i}$ $(1<i<+\infty)$, we have
\begin{equation}\notag
|\lambda_{i}-\lambda_{i,j}^{h}|\leq Ch^{2}\ \ as\ \ \lambda_{i,j}^{h}\to \lambda_{i},\ \ \ \ j=1,2,...,m_{i},\ \
\end{equation}
where C is independent of $h$ but not $\lambda_{i}$ and $m_{i}$ is the algebraic multiplicity of $\lambda_{i}$. Moreover,
\begin{equation}\notag
\theta(V_{\lambda_{i}},M_{h}(\lambda_{i}))\leq Ch,
\end{equation}
where $M_{h}(\lambda_{i})=\oplus_{j=1}^{m_{i}}V_{\lambda^{h}_{i,j}}$, $dim(V_{\lambda_{i}})=m_{i}$, the notation $V_{\lambda}$ denotes the eigenvector space corresponding to the eigenvalue $\lambda$, the $\theta(M,N)$ denotes the gap between subspace $M\subset \bm{H}_{0}(\bm{curl};\Omega)$ and subspace $N \subset \bm{H}_{0}(\bm{curl};\Omega)$, i.e.,
\begin{equation}\notag
\theta(M,N)=\max\{\widetilde{\theta}(M,N),\widetilde{\theta}(N,M)\}
\end{equation}
\begin{equation}\notag
\widetilde{\theta}(M,N)=\sup_{||\bm{u}||_{b}=1,\ \bm{u}\in M}\inf_{\bm{v}\in N}||\bm{u}-\bm{v}||_{b},
\end{equation}
or
\begin{equation}\notag
\widetilde{\theta}(M,N)=\sup_{||\bm{u}||_{c}=1,\ \bm{u}\in M}\inf_{\bm{v}\in N}||\bm{u}-\bm{v}||_{c},
\end{equation}
where the norm $||\cdot||_{c}:=\sqrt{a(\cdot,\cdot)+b(\cdot,\cdot)}$ is equivalent to the norm $||\cdot||_{\bm{curl}}$ in $\bm{H}_{0}(\bm{curl};\Omega)$.
\end{theorem}
\begin{remark}\label{gap}
In general $Banach$ Spaces, the gap between subspace $M$ and subspace $N$  does not construct a distance due to the fact that it does not satisfy the triangle inequality but does satisfy
\begin{equation}\notag
\theta(M,N)\leq d(M,N)\leq 2\theta(M,N),
\end{equation}
where the $d(M,N)$ denotes the Hausdorff distance between subspace $M$ and subspace $N$. For interested readers, please refer to \cite{Kato} for details on the distinction between the gap and Hausdorff distance.
\end{remark}
\par From Theorem 2.1 and Remark $\ref{gap}$, we obtain the following corollary.
\begin{corollary}\label{priorecorollary}
Under the assumption of Theorem $\ref{prioreTheorem}$, let $(\lambda_{1}^{h},\bm{u}_{1}^{h})$ be the first eigen-pair of $\eqref{discreteH0curl}$, $||\bm{u}_{1}^{h}||_{b}=1$ $( ||\bm{u}_{1}^{h}||_{c}=1 )$. Then there exists $h_{0}>0$ such that when $0<h<h_{0}$, it holds that
\begin{equation}\label{prioreigenvalue}
|\lambda_{1}-\lambda_{1}^{h}|\leq Ch^{2},
\end{equation}
and there exists $\bm{u}_{1}\in M(\lambda_{1})$, $||\bm{u}_{1}||_{b}=1$ \ $( ||\bm{u}_{1}||_{c}=1 )$, such that
\begin{equation}\label{prioreigenector0}
||\bm{u}_{1}-\bm{u}_{1}^{h}||_{b}\leq Ch,\ \ \ \  ( ||\bm{u}_{1}-\bm{u}_{1}^{h}||_{c}\leq Ch ).
\end{equation}
\end{corollary}

\section{The two-level PHJD method}
\subsection{Domain decomposition}
\par In this subsection, we shall introduce the domain decomposition and the corresponding subspaces decomposition.
\par Let the coarse quasi-uniform triangulation be $\tau_{H}:=\{\Omega_{i}\}_{i=1}^{N}$, where $H:=max\{H_{i}\ |\ i=1,2,...,N\}$ and $H_{i}:=diam(\Omega_{i})$. The fine shaped-regular and quasi-uniform triangulation is obtained by subdividing $\tau_{H}$ and we denote it by $\tau_{h}$. We may construct the edge element spaces $E_{H}(\Omega_{H})\subset E_{h}(\Omega_{h})$ on $\tau_{H}$ and $\tau_{h}$ but it is well known that $E^{0}_{H}(\Omega_{H};\epsilon_{r})\not\subset E^{0}_{h}(\Omega_{h};\epsilon_{r})$. To get the overlapping subdomains $(\Omega_{i}^{'},\ 1\leq i\leq N)$, we enlarge the subdomains $\Omega_{i}$ by adding fine elements inside $\Omega_{h}$ layer by layer such that $\partial \Omega_{i}^{'}$ does not cut through any fine element. To measure the overlapping width between neighboring subdomains, we define the $\delta_{i}:=dist(\partial\Omega_{i}\setminus\partial\Omega,\partial\Omega_{i}^{'}\setminus\partial\Omega)$ and denote $\delta:=min\{\delta_{i}\ |\ i=1,2,...,N\}$. We also assume that $H_{i}$ is the diameter of the $\Omega_{i}^{'}$.
\par The local subspaces may be defined by
\begin{equation}\label{Vi}
V^{(i)}:=\{\bm{v}_{h}\in E_{h}(\Omega_{h})\ |\ \bm{v}_{h}(\bm{x})=\bm{0}, \ \ \bm{x}\in \Omega\setminus\Omega_{i}^{'} \ \},
\end{equation}
\begin{equation}\label{Vi0}
V_{0}^{(i)}:=\{\bm{v}^{(i)}_{h}\in V^{(i)}\ |\ b(\bm{v}_{h}^{(i)},\nabla{p}_{h})=0,\ \ \ \ p_{h}\in S^{(i)}_{h} \ \},
\end{equation}
which are associated with the local fine mesh in $\Omega_{i}^{'}\ (i=1,2,...,N)$, where $S^{(i)}_{h}:=\{p_{h}\in S_{h}(\Omega_{h})\ |\ p_{h}(\bm{x})=0\ \ \bm{x}\in \Omega\setminus\Omega_{i}^{'} \}$. We denote $Z^{(i)}:V^{(i)}\to V_{0}^{(i)}$ to be the $b(\cdot,\cdot)$-orthogonal projection and let $K^{(i)}:=I-Z^{(i)}$.
\begin{assumption}\label{assumption1}
The partition $\{\Omega_{i}^{'}\}_{i=1}^{N}$ may be colored using at most $N_{0}$ colors, in such a way that subdomains with the same color are disjoint.\ $N_{0}$ is independent of the number of subdomains $N$.
\end{assumption}
\par According to the Assumption 1, we know that if $x\in \Omega$, then it belongs to at most $N_{0}$ subdomains in $\{\Omega_{i}^{'}\}_{i=1}^{N}$. Besides, we obtain a partition of  unity and then there exists a family of functions $\{\theta_{i}\}_{i=1}^{N}$, which are continuous piecewise linear polynomials, satisfy the following properties (see \cite{ToselliM}):
\begin{equation}\label{unitypartition}
supp(\theta_{i})\subset \overline{\Omega_{i}^{'}},\ \ \ \
0 \leq \theta_{i}\leq 1,\ \ \ \
\sum_{i=1}^{N}\theta_{i}(\bm{x})=1,\ \ \bm{x}\in \Omega,\ \ \ \
||\nabla{\theta_{i}}||_{0,\infty,\Omega_{i}^{'}}\leq \frac{C}{\delta}.
\end{equation}

  The strengthened Cauchy-Schwarz inequality is true over the local subspaces $V^{(i)}$, i.e., for $\bm{v}^{(i)}\in V^{(i)}$ and $\bm{v}^{(j)}\in V^{(j)}$, $1\leq i,j\leq N$, there exists $0\leq \eta_{ij}\leq 1$ such that
\begin{equation}
|b(\bm{v}^{(i)},\bm{v}^{(j)})|\leq \eta_{ij}\sqrt{b(\bm{v}^{(i)},\bm{v}^{(i)})}\sqrt{b(\bm{v}^{(j)},\bm{v}^{(j)})}.
\end{equation}
Let $\rho(\Lambda)$ be the spectral radius of the matrix $(\eta_{ij})_{1\leq i,j\leq N}$ and we have (see \cite{ToselliM})
\begin{lemma}\label{strengthenedCauchySchwarzinequality}
Let $\Lambda=(\eta_{ij})_{1\leq i,j\leq N}$. If Assumption 1 holds, then
\[\rho(\Lambda)\leq N_{0}.\]
Moreover, for any $\bm{v}^{(i)}\in V^{(i)}, \bm{v}^{(j)}\in V^{(j)}\ (i,j=1,2,...,N)$,
\begin{equation}\label{strengthened}
\sum_{i,j=1}^{N}b(\bm{v}^{(i)},\bm{v}^{(j)})\leq N_{0}\sum_{i=1}^{N}b(\bm{v}^{(i)},\bm{v}^{(i)}),\ \ \ \sum_{i,j=1}^{N}a(\bm{v}^{(i)},\bm{v}^{(j)})\leq N_{0}\sum_{i=1}^{N}a(\bm{v}^{(i)},\bm{v}^{(i)}).
\end{equation}
\end{lemma}

\subsection{The Jacobi-Davidson method}
\par In this subsection, we focus on the Jacobi-Davidson method proposed by Sleijpen and Vorst (see \cite{Sleijpen}) which is an efficient algebraic method for solving eigenvalue problems. The idea of Jacobi-Davidson method is that one may first solve the Jacobi correction equation in the orthogonal complement space of the current approximation of the eigenvector and then update the new iterative solution in the expanding Davidson subspace. It suggests to compute the correction variable $e$ in the $b(\cdot,\cdot)$-orthogonal complement space of the current approximation $\bm{u}^{k}$ and expand the subspace, in which the new eigen-pair approximation $(\lambda^{k+1},\bm{u}^{k+1})$ may be found.
\par Let $Q_{\bm{u}^{k}}:E_{h}(\Omega_{h})\to span\{\bm{u}^{k}\}$ be $b(\cdot,\cdot)$-orthogonal projection and we present the detailed Jacobi-Davidson algorithm as follows:

\begin{table}[H]
\centering
\begin{tabular}{p{15cm}}
\hline
\hline
\textbf{Algorithm 1} The Jacobi-Davidson algorithm \\
\hline
1. Given the initial approximation  $(\lambda^{1},\bm{u}^{1})$ of the first eigen-pair, $W^{1}=span\{\bm{u}^{1}\}$ and stopping\\
\ \ \ \ tolerance $tol$.\\
2. For $k=$1, 2, 3, ..., denote $\bm{r}^{k}=\lambda^{k}\bm{u}^{k}-A^{h}\bm{u}^{k}$, solving the Jacobi correction equation:
\begin{equation}\label{Jacobicorrtionequation1}
(I-Q_{\bm{u}^{k}})(A^{h}-\lambda^{k})(I-Q_{\bm{u}^{k}})\bm{e}^{k+1}=\bm{r}^{k},\ \ b(\bm{e}^{k+1},\bm{u}^{k})=0.
\end{equation}
3. Minimizing the Rayleigh quotient in $W^{k+1}$\\
\begin{equation}\label{updateDavidson}
\bm{u}^{k+1}=\arg\min_{\bm{0}\ne\bm{v}\in W^{k+1}}Rq(\bm{v}),
\end{equation}
\ \ \ \ where $W^{k+1}=W^{k}+span\{\bm{e}^{k+1}\},\ \lambda^{k+1}=Rq(\bm{u}^{k+1}),$ where $Rq(\bm{v})=\frac{a(\bm{v},\bm{v})}{b(\bm{v},\bm{v})}$. \\
4. If $||\bm{r}^{k}||_{b}<tol$ or $|\lambda^{k+1}-\lambda^{k}|<tol$, then return $(\lambda^{k+1},\bm{u}^{k+1})$, otherwise goto step 2.\\
\hline
\hline
\end{tabular}
\end{table}
\par For the above algorithm, it is easy to see that $\eqref{Jacobicorrtionequation1}$ is equivalent to the following formulation
\begin{equation}
b((A^{h}-\lambda^{k})\bm{e}^{k+1},\bm{v})=b(\bm{r}^{k},\bm{v})\ \ \ \ \bm{v}\in span\{\bm{u}^{k}\}^{\perp}.
\end{equation}
The coefficient operator of the Jacobi correction equation is $A^{h}-\lambda^{k}$, which is indefinite and nearly singular as the approximation $\lambda^{k} \to \lambda_{1}^{h}$. From the fast and parallel computing point of view, we may try to design some efficient preconditioners for $A^{h}-\lambda^{k}$. Assume $B^{k}_{h}\cong A^{h}-\lambda^{k}$, then we may get the preconditioned correction equation:
\begin{equation}\notag
b(B^{k}_{h}\bm{e}^{k+1},\bm{v})=b(\bm{r}^{k},\bm{v})\ \ \ \ \forall\ \bm{v}\in (span\{\bm{u}^{k}\})^{\perp}.
\end{equation}
The correction variable $\bm{e}^{k+1}$ may be obtained by
\[ \bm{e}^{k+1}=(B^{k}_{h})^{-1}\bm{r}^{k}+\beta^{k} (B^{k}_{h})^{-1}\bm{u}^{k},  \]
where $\beta^{k}=-\frac{b((B^{k}_{h})^{-1}\bm{r}^{k},\bm{u}^{k})}{b((B^{k}_{h})^{-1}\bm{u}^{k},\bm{u}^{k})}$ is so-called orthogonal parameter.

\subsection{The two-level PHJD algorithm based on a domain decomposition}
\par Next, we shall propose our two-level preconditioned Helmholtz-Jacobi-Davidson (PHJD) algorithm.
In our case, the preconditioner $(B_{h}^{k})^{-1}$ may be chosen as
\begin{equation}\label{Preconditioner}
(B_{h}^{k})^{-1}:=(B^{k}_{0})^{-1}Q^{H}+\sum_{i=1}^{N}(B^{k}_{i})^{-1}Q^{(i)},
\end{equation}
where $Q^{H}:L^{2}(\Omega)^{3}\to E_{H}(\Omega_{H}),\ Q^{(i)}:E_{h}(\Omega_{h})\to V^{(i)}$ denote the $b(\cdot,\cdot)$-orthogonal projections, $B^{k}_{0}$ denotes $A^{H}-\lambda^{k}$ and $B^{k}_{i}(i=1,2,...,N)$ denotes $(A^{h}-\lambda^{k})|_{V^{(i)}}$, i.e., they satisfy that
for any $\bm{v}^{(i)},\bm{w}^{(i)}\in V^{(i)}$,
\begin{equation}
b(B^{k}_{i}\bm{v}^{(i)},\bm{w}^{(i)})=b((A^{h}-\lambda^{k})|_{V^{(i)}}\bm{v}^{(i)},\bm{w}^{(i)})=b((A^{h}-\lambda^{k})\bm{v}^{(i)},\bm{w}^{(i)}).
\end{equation}
Combining the scaling argument, the Poincar$\acute{e}$ inequality and inverse estimate, it is easy to see that $B^{k}_{i}$ is symmetric about $b(\cdot,\cdot)$ and
we may obtain the following properties:
\begin{equation}\label{Bik}
\lambda_{min}(B^{k}_{i}|_{V_{0}^{(i)}})=O(H^{-2}),\ \ \ \ \ \lambda_{max}(B^{k}_{i}|_{V_{0}^{(i)}})=O(h^{-2}),\ \ \ i=1,2,...,N.
\end{equation}
\begin{table}[H]
\centering
\begin{tabular}{p{15cm}}
\hline
\hline
\textbf{Algorithm 2} The two-level PHJD algorithm\\
\hline
1. Given the initial approximation $(\lambda^{1},\bm{u}^{1})$ of the first eigen-pair.\\
\begin{equation}\label{CoarseProblem1}
A^{H}\bm{u}_{1}^{H}=\lambda_{1}^{H}\bm{u}_{1}^{H},\ \ \ \ ||\bm{u}_{1}^{H}||_{b}=1.
\end{equation}
\begin{equation}\notag
     \begin{cases}\label{Helmholtz1}
         \text{Find $p^{0}_{h}\in S_{h}(\Omega_{h})$, such that}\\
          b(\nabla{p}^{0}_{h},\nabla{q}_{h})=b(\bm{u}_{1}^{H},\nabla{q}_{h})\ \ \ \ \forall\ q_{h}\in S_{h}(\Omega_{h}).
     \end{cases}
\end{equation}
\ \ \ \ Let $\bar{\bm{u}}^{1}=\bm{u}_{1}^{H}-\nabla{p}^{0}_{h},\ \ \bm{u}^{1}=\frac{\bar{\bm{u}}^{1}}{||\bar{\bm{u}}^{1}||_{b}},\ \lambda^{1}=Rq(\bm{u}^{1}),\ W^{1}=span\{\bm{u}^{1}\}.$\\
2. Solving preconditioned Jacobi correction equation, i.e., for $k=$1, 2, 3, ...,\\ \ \ \ \  denote $\bm{r}^{k}=\lambda^{k}\bm{u}^{k}-A^{h}\bm{u}^{k}$ and let
\begin{equation}\label{CorrectionEq1}
b(B_{h}^{k}\bm{e}^{k+1},\bm{v})=b(\bm{r}^{k},\bm{v})\ \ \ \ \forall\ \bm{v}\in (span\{\bm{u}^{k}\})^{\perp}.
\end{equation}
\ \ \ \ Then define
\begin{equation}\label{Correction}
\bm{e}^{k+1}:=(B_{h}^{k})^{-1}\bm{r}^{k}+\beta^{k}(B_{h}^{k})^{-1}\bm{u}^{k},
\end{equation}
\ \ \ \ where $\beta^{k}=-\frac{b((B_{h}^{k})^{-1}\bm{r}^{k},\bm{u}^{k})}{b((B_{h}^{k})^{-1}\bm{u}^{k},\bm{u}^{k})}.$\\
3. Solving the Helmholtz projection system,
\begin{equation}
     \begin{cases}\label{Helmholtz2}
           \text{Find $p^{k}_{h}\in S_{h}(\Omega_{h})$, such that}\\
          b(\nabla{p}^{k}_{h},\nabla{q}_{h})=b(\bm{e}^{k+1},\nabla{q}_{h})\ \ \ \ \forall\ q_{h}\in S_{h}(\Omega_{h}).
     \end{cases}
\end{equation}
\ \ \ \ Let $\bm{t}^{k+1}=\bm{e}^{k+1}-\nabla{p}^{k}_{h}$.\\
4. Minimizing the Rayleigh quotient in $W^{k+1},$\\
\begin{equation}\label{Davidsonexpand}
\bm{u}^{k+1}=\arg\min_{\bm{0}\ne\bm{v}\in W^{k+1}}Rq(\bm{v}),
\end{equation}
\ \ \ \ where $W^{k+1}=W^{k}+span\{\bm{t}^{k+1}\}, \lambda^{k+1}=Rq(\bm{u}^{k+1}).$\\
5. If $||\bm{r}^{k}||_{b}<tol$ or $|\lambda^{k+1}-\lambda^{k}|<tol$, then return $(\lambda^{k+1},\bm{u}^{k+1})$, otherwise goto step 2\\
\hline
\hline
\end{tabular}
\end{table}
\begin{remark}
 We are interested in the simply connected domain $\Omega$ and we may handle the kernel of the $\bm{curl}$ operator by solving the Poisson equation based on the de Rham Complex. If the domain $\Omega$ equipped with the nontrivial topology is considered, we need to consider the harmonic form space $\mathcal{\bm{H}}^{1}$ which is isomorphic to the de Rham co-homology group as well as the quotient space $ker(\bm{curl})/\nabla{H_{0}^{1}(\Omega)}$ whose dimension is the corresponding Betti number(see, e.g., \cite{Arnold,Boffi}).
\end{remark}
\begin{remark}\label{Hs}
For the case $\lambda_{1}^{H}\leq \lambda_{1}^{h}$, it is obvious that $(B_{0}^{k})^{-1}$ is well-defined. For the case $\lambda_{1}^{h}\leq \lambda_{1}^{H}$, the $B_{0}^{k}$ may be singular in our iterative procedure. In this case, we only need to refine the initial grid in step 1 in Algorithm 2 to obtain $\lambda^{1}$, which can ensure that $(B_{0}^{k})^{-1}$ is well-defined.
\end{remark}
\section{Convergence analysis}
\par In this section, we focus on giving a convergence analysis of the two-level PHJD method. First, we present some useful lemmas in our main proof. The first lemma illustrates that as the coarse grid size $H$ is sufficiently small, the distance in the sense of $||\cdot||_{b}$-norm between the first approximation of the eigenvector $\bm{u}^{1}$ with $\bm{u}^{H}_{1}$ is sufficiently small too, as well as the distance between the first approximation of the eigenvalue $\lambda^{1}$ with $\lambda^{H}_{1}$. In fact, the proof of the first inequality in the first lemma essentially needs the Hodge operator lemma (see \cite{Hiptmair1}). You may see a complete proof in \cite{chenlong}. So next, we only prove the second and the third inequality in Lemma $\ref{Lemmalambda1lambda1H}$.
\begin{lemma}\label{Lemmalambda1lambda1H}
There exists a constant C such that
\begin{equation}\label{nablaph}
||\nabla{p}^{0}_{h}||_{b}=||\bm{u}_{1}^{H}-\bar{\bm{u}}^{1}||_{b}\leq C{H},
\end{equation}
\begin{equation}\label{nablaph1}
||\bm{u}_{1}^{H}-\bm{u}^{1}||_{b}\leq C{H},
\end{equation}
\begin{equation}\label{lambda1lambda1H}
0\leq\lambda^{1}-\lambda_{1}^{H}\leq C{H}^{2},
\end{equation}
where the notations $p^{0}_{h},\bm{u}_{1}^{H},\lambda_{1}^{H},\bar{\bm{u}}^{1},\bm{u}^{1},\lambda^{1}$ were defined in the above two-level PHJD algorithm.
\end{lemma}
\begin{proof}
We first prove $\eqref{nablaph1}$. Due to $\bar{\bm{u}}^{1}=\bm{u}_{1}^{H}-\nabla{p}_{h}^{0}$ and $\eqref{nablaph}$, we have
\begin{align*}
||\bm{u}_{1}^{H}-\bm{u}^{1}||_{b}&\leq ||\bar{\bm{u}}^{1}-\bm{u}^{1}||_{b}+||\bm{u}_{1}^{H}-\bar{\bm{u}}^{1}||_{b}=1-||\bar{\bm{u}}^{1}||_{b}+||\bm{u}_{1}^{H}-\bar{\bm{u}}^{1}||_{b}\\
&\leq 1-(||\bm{u}_{1}^{H}||_{b}-||\bm{u}_{1}^{H}-\bar{\bm{u}}^{1}||_{b})+||\bm{u}_{1}^{H}-\bar{\bm{u}}^{1}||_{b}\leq C{H}.
\end{align*}
Next, we compute the formula $\lambda^{1}-\lambda_{1}^{H}$ directly. Owing to $\bar{\bm{u}}^{1}=\bm{u}_{1}^{H}-\nabla{p}_{h}^{0}$ and $\bm{u}^{1}=\frac{\bar{\bm{u}}^{1}}{||\bar{\bm{u}}^{1}||_{b}}$, we have
\begin{align}
\lambda^{1}-\lambda_{1}^{H}&=\frac{a(\bm{u}^{1},\bm{u}^{1})}{b(\bm{u}^{1},\bm{u}^{1})}-a(\bm{u}_{1}^{H},\bm{u}_{1}^{H})
=\frac{a(\bar{\bm{u}}^{1},\bar{\bm{u}}^{1})}{b(\bar{\bm{u}}^{1},\bar{\bm{u}}^{1})}-a(\bm{u}_{1}^{H},\bm{u}_{1}^{H})\notag\\
&=(\frac{1}{1-b(\nabla{p}^{0}_{h},\nabla{p}^{0}_{h})}-1)a(\bm{u}_{1}^{H},\bm{u}_{1}^{H})\label{pp1}.
\end{align}
Because of $\eqref{nablaph}$, we obtain
\begin{align}
\lambda^{1}-\lambda_{1}^{H}
&\leq\frac{C{H}^{2}}{1-C{H}^{2}}a(\bm{u}_{1}^{H},\bm{u}_{1}^{H})\\
&\leq C{H}^{2}\notag.
\end{align}
Moreover, from $\eqref{pp1}$, we know that  $\lambda^{1}-\lambda_{1}^{H}\geq0$. \qed
\end{proof}
\par We now illustrate that all the norms $||\cdot||_{a},||\cdot||_{E^{k}},||\cdot||_{E^{h}},||\cdot||_{E}$ defined in $M^{\perp}_{h}(\lambda_{1})$ are equivalent, i.e., these topologies induced by these norms are homeomorphic.
\begin{lemma}\label{normequivalence}
The following bilinear forms construct the inner product in $M^{\perp}_{h}(\lambda_{1})$,
\begin{align*}
&\text{$(\bm{v}_{2}^{h},\bm{v}_{2}^{h})_{E^{k}}:=a(\bm{v}_{2}^{h},\bm{v}_{2}^{h})-\lambda^{k}b(\bm{v}_{2}^{h},\bm{v}_{2}^{h})$}\ \ \ \   \forall\ \bm{v}_{2}^{h}\in M^{\perp}_{h}(\lambda_{1}),\\
&\text{$(\bm{v}_{2}^{h},\bm{v}_{2}^{h})_{E^{h}}:=a(\bm{v}_{2}^{h},\bm{v}_{2}^{h})-\lambda_{1}^{h}b(\bm{v}_{2}^{h},\bm{v}_{2}^{h})$}\ \ \ \   \forall\ \bm{v}_{2}^{h}\in M^{\perp}_{h}(\lambda_{1}),\\
&\text{$(\bm{v}_{2}^{h},\bm{v}_{2}^{h})_{E}:=a(\bm{v}_{2}^{h},\bm{v}_{2}^{h})-\lambda_{1}b(\bm{v}_{2}^{h},\bm{v}_{2}^{h})$}\ \ \ \ \
\forall\ \bm{v}_{2}^{h}\in M^{\perp}_{h}(\lambda_{1}).
\end{align*}
 Moreover, the norms $||\cdot||_{E^{k}},||\cdot||_{E^{h}},||\cdot||_{E}$ induced by above inner products and the norm $||\cdot||_{a}$ in $M^{\perp}_{h}(\lambda_{1})$ are equivalent.
\end{lemma}
\begin{proof}
   Combining $\eqref{lambda1lambda1H}$ and the fact that the Rayleigh quotient $Rq(\bm{v})$ is a monotonic decreasing function with expanding Davidson space $W^{k}$, for sufficiently small $H$, we have
\begin{equation}\notag
\lambda_{2}^{h}-\lambda^{k}\geq C,
\end{equation}
where the constant $C$ is independent of $h, H$. Furthermore, for any $\bm{v}_{2}^{h}(\ne \bm{0})\in M^{\perp}_{h}(\lambda_{1})$, we have
\begin{align}
(\bm{v}_{2}^{h},\bm{v}_{2}^{h})_{E^{k}}&=a(\bm{v}_{2}^{h},\bm{v}_{2}^{h})-\lambda^{k}b(\bm{v}_{2}^{h},\bm{v}_{2}^{h})\notag\\
&\geq (\lambda_{2}^{h}-\lambda^{k})b(\bm{v}_{2}^{h},\bm{v}_{2}^{h}) \label{line2}\\
&\geq C||\bm{v}_{2}^{h}||^{2}_{b}\notag>0.
\end{align}
Hence, we know that $(\cdot,\cdot)_{E^{k}}$ defines an inner product in $M^{\perp}(\lambda_{1})$.
For any $\bm{v}_{2}^{h}\in M^{\perp}_{h}(\lambda_{1})$, we have
\[a(\bm{v}_{2}^{h},\bm{v}_{2}^{h})\geq a(\bm{v}_{2}^{h},\bm{v}_{2}^{h})-\lambda^{k}b(\bm{v}_{2}^{h},\bm{v}_{2}^{h})=(\bm{v}_{2}^{h},\bm{v}_{2}^{h})_{E^{k}}. \]
Owing to $\eqref{line2}$, we know $(\bm{v}_{2}^{h},\bm{v}_{2}^{h})_{E^{k}}\geq (\lambda_{2}^{h}-\lambda^{k})b(\bm{v}_{2}^{h},\bm{v}_{2}^{h})$. Hence,
\[(\bm{v}_{2}^{h},\bm{v}_{2}^{h})_{E^{k}}=a(\bm{v}_{2}^{h},\bm{v}_{2}^{h})-\lambda^{k}b(\bm{v}_{2}^{h},\bm{v}_{2}^{h})\geq a(\bm{v}_{2}^{h},\bm{v}_{2}^{h})-\frac{\lambda^{k}}{\lambda_{2}^{h}-\lambda^{k}}(\bm{v}_{2}^{h},\bm{v}_{2}^{h})_{E^{k}}, \]
that is
\[ a(\bm{v}_{2}^{h},\bm{v}_{2}^{h})\leq (1+\frac{\lambda^{k}}{\lambda_{2}^{h}-\lambda^{k}})(\bm{v}_{2}^{h},\bm{v}_{2}^{h})_{E^{k}}.\]
It is easy to see that as $h\to 0^{+},$
\[1+\frac{\lambda^{k}}{\lambda_{2}^{h}-\lambda^{k}}=\frac{\lambda_{2}^{h}}{\lambda_{2}^{h}-\lambda^{k}}
\to\frac{\lambda_{2}}{\lambda_{2}-\lambda^{k}}, \]
$i.e.$, there exists a constant $\eta$ which is independent of $h,H$, such that
$0<\frac{\lambda_{2}}{\lambda_{2}-\lambda^{k}}-\eta\leq \frac{\lambda_{2}^{h}}{\lambda_{2}^{h}-\lambda^{k}}\leq \frac{\lambda_{2}}{\lambda_{2}-\lambda^{k}}+\eta $. Other conclusions in this lemma can be proved by a similar argument. \qed
\end{proof}

\par The following lemma, which may be proved by using the same technique as in the case of nodal spaces, may be regarded as the 'bridge' between the coarse space and the fine space. (see \cite{ToselliM,Toselli} and \cite{Bramble1})
\begin{lemma}\label{QH}
Let $\tau_{H}$ be a shape-regular and quasi-uniform triangulation. Then
\begin{equation}\label{curlQH}
||\bm{curl} Q^{H}\bm{u}||_{b}\leq C|\bm{u}|_{1,\Omega} \ \ \ \ \ \ \forall\ \bm{u}\in H^{1}(\Omega)^{3},
\end{equation}
\begin{equation}\label{curlQH1}
||\bm{u}-Q^{H}\bm{u}||_{b}\leq CH|\bm{u}|_{1,\Omega} \ \ \ \ \ \ \forall\ \bm{u}\in H^{1}(\Omega)^{3}.
\end{equation}
\end{lemma}

\par Furthermore, we have
\begin{lemma}\label{LemmaK0hQ2H}
For any $\bm{v}^{h}\in E_{h}(\Omega_{h})$, it holds that
\begin{equation}\label{K0Q1H}
||K_{0}^{h}Q_{1}^{H}\bm{v}^{h}||_{b}\leq CH||Q_{1}^{H}\bm{v}^{h}||_{b},
\end{equation}
\begin{equation}\label{Q2Q1H}
||Q_{2}^{h}Q_{1}^{H}\bm{v}^{h}||_{b}\leq CH||Q_{1}^{H}\bm{v}^{h}||_{b},\ \ and \ \ ||Q_{2}^{h}Q_{1}^{H}\bm{v}^{h}||_{a}\leq CH||Q_{1}^{H}\bm{v}^{h}||_{a},
\end{equation}
\begin{equation}\label{Q2HQ1}
||Q_{2}^{H}Q_{1}^{h}\bm{v}^{h}||_{b}\leq CH||Q_{1}^{h}\bm{v}^{h}||_{b},\ \ and \ \ ||Q_{2}^{H}Q_{1}^{h}\bm{v}^{h}||_{a}\leq CH||Q_{1}^{h}\bm{v}^{h}||_{a},
\end{equation}
where $C$ is independent of $H, h$.
\end{lemma}
\begin{proof}
For any $\bm{v}_{1}^{H}\in M_{H}(\lambda_{1})$, let $\bm{u}_{1}^{H}=\frac{\bm{v}_{1}^{H}}{||\bm{v}_{1}^{H}||_{b}}$. According to the Theorem $\ref{prioreTheorem}$ and Corollary $\ref{priorecorollary}$, we know that there exists a vector $\bm{w}_{1}\in M(\lambda_{1})\ ( ||\bm{w}_{1}||_{b}=1)$ such that
\[||\bm{u}_{1}^{H}-\bm{w}_{1}||_{b}\leq CH.   \]
By multiplying the $||\bm{v}_{1}^{H}||_{b}$ at the both sides of the above inequality, we have
\[||\bm{v}_{1}^{H}-||\bm{v}_{1}^{H}||_{b}\bm{w}_{1}||_{b}\leq CH||\bm{v}_{1}^{H}||_{b}.   \]
Define $\bm{\xi}_{1}:=||\bm{v}_{1}^{H}||_{b}\bm{w}\in M(\lambda_{1})$, we have $||\bm{\xi}_{1}||_{b}=||\bm{v}_{1}^{H}||_{b}$. Moreover,
\begin{equation}\label{K0hQ1H1}
||\bm{v}_{1}^{H}-\bm{\xi}_{1}||_{b}\leq CH||\bm{\xi}_{1}||_{b}.
\end{equation}
For the $\bm{\xi}_{1}\in M(\lambda_{1})$, similarly, we may find a $\bm{v}_{1}^{h}\in M_{h}(\lambda_{1})$ $(||\bm{v}_{1}^{h}||_{b}=||\bm{\xi}_{1}||_{b})$ such that
\begin{equation}\label{K0hQ1H2}
||\bm{\xi}_{1}-\bm{v}_{1}^{h}||_{b}\leq Ch||\bm{\xi}_{1}||_{b}.
\end{equation}
Hence, combining $\eqref{K0hQ1H1},\eqref{K0hQ1H2}$ with the triangle inequality, we obtain
\begin{equation}\label{K0hQ1H3}
||\bm{v}_{1}^{H}-\bm{v}_{1}^{h}||_{b}\leq CH||\bm{v}_{1}^{H}||_{b}.
\end{equation}
For $Q_{1}^{H}\bm{v}^{h}\in M_{H}(\lambda_{1})$, there exists a $\bm{\eta}_{1}^{h}\in M_{h}(\lambda_{1})$ such that
\[  ||Q_{1}^{H}\bm{v}^{h}-\bm{\eta}_{1}^{h}||_{b}\leq CH||Q_{1}^{H}\bm{v}^{h}||_{b}. \]
Then
\begin{align*}
&\ \ \ \ ||K_{0}^{h}Q_{1}^{H}\bm{v}^{h}||_{b}^{2}=b(K_{0}^{h}Q_{1}^{H}\bm{v}^{h},Q_{1}^{H}\bm{v}^{h}-\bm{\eta}_{1}^{h})\\
&=||K_{0}^{h}Q_{1}^{H}\bm{v}^{h}||_{b}||Q_{1}^{H}\bm{v}^{h}-\bm{\eta}_{1}^{h}||_{b}\leq CH||K_{0}^{h}Q_{1}^{H}\bm{v}^{h}||_{b}||Q_{1}^{H}\bm{v}^{h}||_{b}.
\end{align*}
Finally,
\[||K_{0}^{h}Q_{1}^{H}\bm{v}^{h}||_{b}\leq CH||Q_{1}^{H}\bm{v}^{h}||_{b}.\]
Note that the norm $||\cdot||_{c}$ is equivalent to the norm $||\cdot||_{a}$ in the discrete divergence-free space. By corollary $\ref{priorecorollary}$,
we may prove $\eqref{Q2Q1H}$ and $\eqref{Q2HQ1}$ by a similar argument.\qed
\end{proof}
\par   For the convenience of the following convergence analysis, we choose a special case to analyze the error reduction. Let
$\bm{u}^{k}=\arg\min_{\bm{0}\ne\bm{v}\in W^{k}}Rq(\bm{v})$ and $\lambda^{k}=Rq(\bm{u}^{k})$. Next, we minimize the $Rq(\bm{v})$ in $W^{k+1}$ and choose the $\check{\bm{u}}^{k+1}$ in $W^{k+1}$ ($||\check{\bm{u}}^{k+1}||_{b}=1$) to analyze the error reduction, i.e.,
\begin{equation}\label{specialcase1}
\widetilde{\bm{u}}^{k+1}=\bm{u}^{k}+\alpha \bm{t}^{k+1},
\end{equation}
and set
\begin{equation}\label{specialcase2}
\check{\bm{u}}^{k+1}=\frac{\widetilde{\bm{u}}^{k+1}}{||\widetilde{\bm{u}}^{k+1}||_{b}},
\end{equation}
where $\alpha$ is an undetermined parameter depending on the $N_{0}$. By the above analysis, we know that $\lambda^{k+1}\leq \check{\lambda}^{k+1}$, where $\check{\lambda}^{k+1}=Rq(\check{\bm{u}}^{k+1})$.
\par Owing to the Helmholtz projection, we know that $\widetilde{\bm{u}}^{k+1}\in E^{0}_{h}(\Omega_{h};\epsilon_{r})$  and we denote
\begin{equation}\label{u1e2definition}
\bm{u}_{1}^{k}:=Q_{1}^{h}\bm{u}^{k},\ \ \ \ \ \ \bm{e}_{2}^{k}:=-Q_{2}^{h}\bm{u}^{k},\ \ \ \ \ \widetilde{\bm{e}}_{2}^{k}:=-Q_{2}^{h}\widetilde{\bm{u}}^{k}.
\end{equation}
From $\eqref{Correction}$, $\eqref{specialcase1}$, it is easy to see that
\begin{align*}\notag
\widetilde{\bm{e}}_{2}^{k+1}&=-Q_{2}^{h}\{\bm{u}^{k}+\alpha \bm{t}^{k+1}\}=\bm{e}_{2}^{k}-\alpha Q_{2}^{h}\{\bm{e}^{k+1}-\nabla{p}_{h}^{k}\}\\
&=\bm{e}_{2}^{k}-\alpha Q_{2}^{h}\bm{e}^{k+1}=\bm{e}_{2}^{k}-\alpha Q_{2}^{h}\{(B_{h}^{k})^{-1}\bm{r}^{k}+\beta^{k}(B_{h}^{k})^{-1}\bm{u}^{k}\}.
\end{align*}
We substitute the expression of the preconditioner $\eqref{Preconditioner}$ into the above formula, and then obtain
\begin{align*}
\widetilde{\bm{e}}_{2}^{k+1}&=\bm{e}_{2}^{k}-\alpha Q_{2}^{h}\{(B_{h}^{k})^{-1}\bm{r}^{k}+\beta^{k}(B_{h}^{k})^{-1}\bm{u}^{k}\}\\
&=\bm{e}_{2}^{k}-\alpha Q_{2}^{h}\{(B^{k}_{0})^{-1}Q^{H}\bm{r}^{k}+\sum_{i=1}^{N}(B^{k}_{i})^{-1}Q^{(i)}\bm{r}^{k}\\
&\ \ \ \ +\beta^{k}(B^{k}_{0})^{-1}Q^{H}\bm{u}^{k}+\beta^{k}\sum_{i=1}^{N}(B^{k}_{i})^{-1}Q^{(i)}\bm{u}^{k}\}\\
&=\bm{e}_{2}^{k}-\alpha Q_{2}^{h}\{(B^{k}_{0})^{-1}(K_{0}^{H}+Z^{(0)})Q^{H}\bm{r}^{k}+\sum_{i=1}^{N}(B^{k}_{i})^{-1}(K^{(i)}+Z^{(i)})Q^{(i)}\bm{r}^{k}\\
&\ \ \ +\beta^{k}(B^{k}_{0})^{-1}(K_{0}^{H}+Z^{(0)})Q^{H}\bm{u}^{k}+\beta^{k}\sum_{i=1}^{N}(B^{k}_{i})^{-1}(K^{(i)}+Z^{(i)})Q^{(i)}\bm{u}^{k}\},
\end{align*}
where the last equality has used the partitions of the identity operator defined on $E_{H}(\Omega_{H})$ and $V^{(i)}$.
For any $\bm{v}\in E_{h}(\Omega_{h})$, $(B^{k}_{0})^{-1}K_{0}^{H}Q^{H}\bm{v}\in \nabla{S_{H}(\Omega_{H})}\subset \nabla{S_{h}(\Omega_{h})}$, $(B^{k}_{i})^{-1}K^{(i)}Q^{(i)}\bm{v} \in \nabla{S^{(i)}_{h}}\subset\nabla{S_{h}(\Omega_{h})}$, so we have
\[Q_{2}^{h}(B^{k}_{0})^{-1}K_{0}^{H}Q^{H}\bm{v}=\bm{0},\ \ \ Q_{2}^{h}(B^{k}_{i})^{-1}K^{(i)}Q^{(i)}\bm{v}=\bm{0}.  \]
Furthermore, by using the partitions of the identity operator defined on $E^{0}_{H}(\Omega_{H};\epsilon_{r})$, we obtain
\begin{align*}
\widetilde{\bm{e}}_{2}^{k+1}&=\bm{e}_{2}^{k}-\alpha Q_{2}^{h}(B^{k}_{0})^{-1}Z^{(0)}Q^{H}\bm{r}^{k}-\alpha Q_{2}^{h}\sum_{i=1}^{N}(B^{k}_{i})^{-1}Z^{(i)}Q^{(i)}\bm{r}^{k}\\
&\ \ \ -\alpha\beta^{k}Q_{2}^{h}(B^{k}_{0})^{-1}Z^{(0)}Q^{H}\bm{u}^{k}-\alpha\beta^{k}Q_{2}^{h}\sum_{i=1}^{N}(B^{k}_{i})^{-1}Z^{(i)}Q^{(i)}\bm{u}^{k}\\
&=\bm{e}_{2}^{k}-\alpha Q_{2}^{h}(B^{k}_{0})^{-1}(Q_{1}^{H}+Q_{2}^{H})Z^{(0)}Q^{H}\bm{r}^{k}-\alpha Q_{2}^{h}\sum_{i=1}^{N}(B^{k}_{i})^{-1}Z^{(i)}Q^{(i)}\bm{r}^{k}\\
&\ \ \ -\alpha\beta^{k}Q_{2}^{h}(B^{k}_{0})^{-1}(Q_{1}^{H}+Q_{2}^{H})Z^{(0)}Q^{H}\bm{u}^{k}-\alpha\beta^{k}Q_{2}^{h}\sum_{i=1}^{N}(B^{k}_{i})^{-1}Z^{(i)}Q^{(i)}\bm{u}^{k},
\end{align*}
which, together with the definition of residual vector $\bm{r}^{k}$ and $\eqref{u1e2definition}$, yields
\begin{align}
\widetilde{\bm{e}}_{2}^{k+1}&=\{\bm{e}_{2}^{k}-\alpha Q_2^{h}(B^{k}_{0})^{-1}Q_{2}^{H}Z^{(0)}Q^{H}(A^{h}-\lambda^{k})\bm{e}_{2}^{k}-\alpha Q_2^{h}\sum_{i=1}^{N}(B^{k}_{i})^{-1}Z^{(i)}Q^{(i)}(A^{h}-\lambda^{k})\bm{e}_{2}^{k}\}\notag\\
&\ \ \ +\{\alpha Q_{2}^{h}(B_{0}^{k})^{-1}Q_{1}^{H}Z^{(0)}Q^{H}(A^{h}-\lambda^{k})\bm{u}^{k}
-\alpha\beta^{k}Q_{2}^{h}(B_{0}^{k})^{-1}Q_{1}^{H}Z^{(0)}Q^{H}\bm{u}^{k}\}\notag\\
&\ \ \ +\{\alpha Q_2^{h}(B^{k}_{0})^{-1}Q_{2}^{H}Z^{(0)}Q^{H}(A^{h}-\lambda^{k})\bm{u}_{1}^{k}+\alpha Q_2^{h}\sum_{i=1}^{N}(B^{k}_{i})^{-1}Z^{(i)}Q^{(i)}(A^{h}-\lambda^{k})\bm{u}_{1}^{k}\notag\\
&\ \ \ -\alpha\beta^{k}Q_2^{h}(B^{k}_{0})^{-1}Q_{2}^{H}Z^{(0)}Q^{H}\bm{u}^{k}-
\alpha\beta^{k}Q_2^{h}\sum_{i=1}^{N}(B^{k}_{i})^{-1}Z^{(i)}Q^{(i)}\bm{u}^{k}\}\notag\\
&:=I_{1}+I_{2}+I_{3},\label{errorsplitting}
\end{align}
here we define
$G^{k}:=I-\alpha Q_2^{h}(B^{k}_{0})^{-1}Q_{2}^{H}Z^{(0)}Q^{H}(A^{h}-\lambda^{k})-\alpha Q_2^{h}\sum_{i=1}^{N}(B^{k}_{i})^{-1}Z^{(i)}Q^{(i)}(A^{h}-\lambda^{k})$
and let $I_{1}:=G^{k}\bm{e}_{2}^{k}$. we call
$I_{1}$ as the error principal term and $G^{k}$ as the error principal operator and $I_{2}$ as the almost counterbalanced term. Noting that $a(\bm{u}^{k},\bm{u}^{k})=\lambda^{k}b(\bm{u}^{k},\bm{u}^{k})$ and using $\eqref{u1e2definition}$, it is easy to prove that
\begin{equation}\label{e2Eku1b}
||\bm{e}_{2}^{k}||^{2}_{E^{k}}=\lambda^{k}b(\bm{u}_{1}^{k},\bm{u}_{1}^{k})-a(\bm{u}_{1}^{k},\bm{u}_{1}^{k})
= (\lambda^{k}-\lambda^{h}_{1})b(\bm{u}_{1}^{k},\bm{u}_{1}^{k}),
\end{equation}
which is useful in the following error estimates.
\par Moreover, we note that $|\lambda^{k}-\lambda_{1}^{h}|\leq CH^{2}$ and $|\lambda^{k}-\lambda_{1}|\leq CH^{2}$. In fact, as the Rayleigh quotient $Rq(\bm{v})$ is a monotonic decreasing function with  expanding the Davidson subspace $W^{k}$, we have $\lambda_{1}^{h}\leq \lambda^{k}\leq \lambda^{1}$. For the case $\lambda_{1}^{H}\leq \lambda_{1}^{h}$, by corollary $\ref{priorecorollary}$ and $\eqref{lambda1lambda1H}$, we know that
\begin{equation}\label{lowerboundestimate}
|\lambda^{k}-\lambda_{1}|\leq |\lambda^{k}-\lambda_{1}^{h}|+|\lambda_{1}^{h}-\lambda_{1}|
\leq |\lambda^{1}-\lambda_{1}^{H}|+|\lambda_{1}^{h}-\lambda_{1}| \leq CH^{2}.
\end{equation}
For the case $\lambda_{1}^{H}>\lambda_{1}^{h}$, by corollary $\ref{priorecorollary}$ and $\eqref{lambda1lambda1H}$, we have
\begin{align}
|\lambda^{k}-\lambda_{1}|&\leq |\lambda^{k}-\lambda_{1}^{h}|+|\lambda_{1}^{h}-\lambda_{1}|\leq |\lambda^{1}-\lambda_{1}^{h}|+|\lambda_{1}^{h}-\lambda_{1}|\notag\\
&\leq |\lambda^{1}-\lambda_{1}^{H}|+|\lambda_{1}^{H}-\lambda_{1}^{h}|+|\lambda_{1}^{h}-\lambda_{1}|\leq CH^{2}.\label{upperboundestimate}
\end{align}
\subsection{Estimate of the error principal term $I_{1}$}
\par In this subsection, we analyze the error principal term $I_{1}$.
\begin{theorem}\label{Gkv2theorem}
For sufficiently small $\alpha$, there exists a constant $C$ such that
\begin{equation}\label{Gkv2hestimate}
||G^{k}\bm{v}_{2}^{h}||_{E^{k}}\leq (1-C\frac{\delta^{2}}{H^{2}})||\bm{v}_{2}^{h}||_{E^{k}}\ \ \ \ \ \forall \ \bm{v}_{2}^{h}\in M_{h}^{\perp}(\lambda_{1}),
\end{equation}
where the constant $C$ is independent of $H,\ \delta$ and $ h$.
\end{theorem}
\par First of all, we give two useful lemmas.
\begin{lemma}\label{Gksympos}
The error principal operator $G^{k}:M_{h}^{\perp}(\lambda_{1})\to M_{h}^{\perp}(\lambda_{1})$ is symmetric in the sense of $(\cdot,\cdot)_{E^{k}}$. Furthermore, if $\alpha$ is sufficiently small, the operator $G^{k}:M_{h}^{\perp}(\lambda_{1})\to M_{h}^{\perp}(\lambda_{1})$ is also positive definite.
\end{lemma}
\begin{proof}
It is easy to see that $(B_{0}^{k})^{-1}$ and $(B_{i}^{k})^{-1},\ (i=1,2,...,N)$ are symmetric in the sense of $b(\cdot,\cdot)$, which result in
\begin{equation}\label{pricipalcoarse}
(Q_2^{h}(B^{k}_{0})^{-1}Q_{2}^{H}Z^{(0)}Q^{H}(A^{h}-\lambda^{k})\bm{v}_{2}^{h},\bm{w}_{2}^{h})_{E^{k}}=
(\bm{v}_{2}^{h},Q_2^{h}(B^{k}_{0})^{-1}Q_{2}^{H}Z^{(0)}Q^{H}(A^{h}-\lambda^{k})\bm{w}_{2}^{h})_{E^{k}},
\end{equation}
and
\begin{equation}\label{pricipalfine}
(Q_2^{h}\sum_{i=1}^{N}(B^{k}_{i})^{-1}Z^{(i)}Q^{(i)}(A^{h}-\lambda^{k})\bm{v}_{2}^{h},\bm{w}_{2}^{h})_{E^{k}}=
(\bm{v}_{2}^{h},Q_2^{h}\sum_{i=1}^{N}(B^{k}_{i})^{-1}Z^{(i)}Q^{(i)}(A^{h}-\lambda^{k})\bm{w}_{2}^{h})_{E^{k}},
\end{equation}
here $\bm{v}_{2}^{h},\bm{w}_{2}^{h}\in M_{h}^{\perp}(\lambda_{1})$.
Combining $\eqref{pricipalcoarse}$, $\eqref{pricipalfine}$ with the definition of $G^{k}$, we obtain
\begin{equation}\label{Gkequ}
(G^{k}\bm{v}_{2}^{k},\bm{w}_{2}^{k})_{E^{k}}=(\bm{v}_{2}^{k},G^{k}\bm{w}_{2}^{k})_{E^{k}},
\end{equation}
which means that $G^{k}$ is symmetric in the sense of $(\cdot,\cdot)_{E^{k}}$.
\par To obtain the second conclusion in this lemma, we define an operator
$T_{0}^{k}:M_{h}^{\perp}(\lambda_{1})\to M_{H}^{\perp}(\lambda_{1})$ as follows:
\begin{equation}\label{T0k}
(T_{0}^{k}\bm{v}_{2}^{h},\bm{v}_{2}^{H})_{E^{k}}=(\bm{v}_{2}^{h},\bm{v}_{2}^{H})_{E^{k}}\ \ \ \ \ \forall\ \bm{v}_{2}^{H}\in M_{H}^{\perp}(\lambda_{1}),
\end{equation}
where $\bm{v}_{2}^{h}\in M^{\perp}_{h}(\lambda_{1})$. For sufficiently small $H$, we know that $\lambda^{k}<\lambda_{2}^{H}$. By the Lax-Milgram theorem in $M_{H}^{\perp}(\lambda_{1})$, we may see that the above definition is meaningful.
Similarly, we may also define operator $T_{i}^{k}:M_{h}^{\perp}(\lambda_{1})\to V_{0}^{(i)}$ $(i=1, 2, ...,N)$ as follows:
\begin{equation}\label{Tik}
(T_{i}^{k}\bm{v}_{2}^{h},\bm{v}_{0}^{(i)})_{E^{k}}=(\bm{v}_{2}^{h},\bm{v}_{0}^{(i)})_{E^{k}}\ \ \ \ \ \forall\ \bm{v}_{0}^{(i)}\in V_{0}^{(i)},
\end{equation}
where $\bm{v}_{2}^{h}\in M_{h}^{\perp}(\lambda_{1})$. Because the Poincar$\acute{e}$ inequality holds in local fine spaces $V_{0}^{(i)},\ \ i=1,2,...,N$, the above definitions are also meaningful.
\par It is easy to check that $T_{0}^{k}=(B^{k}_{0})^{-1}Q_{2}^{H}Z^{(0)}Q^{H}(A^{h}-\lambda^{k})$. In fact, for any $\bm{v}_{2}^{h},\bm{v}_{2}^{H}$, we have
\begin{equation}\notag
b(B^{k}_{0}T_{0}^{k}\bm{v}_{2}^{h},\bm{v}_{2}^{H})=b((A^{H}-\lambda^{k})T_{0}^{k}\bm{v}_{2}^{h},\bm{v}_{2}^{H})
=(T_{0}^{k}\bm{v}_{2}^{h},\bm{v}_{2}^{H})_{E^{k}}
=(\bm{v}_{2}^{h},\bm{v}_{2}^{H})_{E^{k}}
\end{equation}
\begin{equation}\notag
=b((A^{h}-\lambda^{k})\bm{v}_{2}^{h},\bm{v}_{2}^{H})\\
=b(Q_{2}^{H}Z^{(0)}Q^{H}(A^{h}-\lambda^{k})\bm{v}_{2}^{h},\bm{v}_{2}^{H}).
\end{equation}
Similarly, $T_{i}^{k}=(B^{k}_{i})^{-1}Z^{(i)}Q^{(i)}(A^{h}-\lambda^{k})$.
\par Hence, for any $\bm{v}_{2}^{h}\in M^{\perp}(\lambda_{1})$, we have
\begin{align}
(G^{k}\bm{v}_{2}^{h},\bm{v}_{2}^{h})_{E^{k}}
&=(\bm{v}_{2}^{h},\bm{v}_{2}^{h})_{E^{k}}-\alpha(Q_2^{h}T_{0}^{k}\bm{v}_{2}^{h},\bm{v}_{2}^{h})_{E^{k}}-
\alpha(Q_2^{h}\sum_{i=1}^{N}T_{i}^{k}\bm{v}_{2}^{h},\bm{v}_{2}^{h})_{E^{k}}\notag \\
&=||\bm{v}_{2}^{h}||^{2}_{E^{k}}-\alpha(T^{k}_{0}\bm{v}_{2}^{h},\bm{v}_{2}^{h})_{E^{k}}-
\alpha\sum_{i=1}^{N}(T_{i}^{k}\bm{v}_{2}^{h},\bm{v}_{2}^{h})_{E^{k}}\notag \\
&=||\bm{v}_{2}^{h}||^{2}_{E^{k}}-\alpha||T^{k}_{0}\bm{v}_{2}^{h}||^{2}_{E^{k}}
-\alpha\sum_{i=1}^{N}||T_{i}^{k}\bm{v}_{2}^{h}||^{2}_{E^{k}}.\label{Gkpri}
\end{align}
For the second term of $\eqref{Gkpri}$, owing to the definition of $Q_{2}^{h}$ and Cauchy-Schwarz inequality, we have
\begin{equation}\label{Tok1}
||T^{k}_{0}\bm{v}_{2}^{h}||^{2}_{E^{k}}=(T^{k}_{0}\bm{v}_{2}^{h},\bm{v}_{2}^{h})_{E^{k}}
=(Q_{2}^{h}T^{k}_{0}\bm{v}_{2}^{h},\bm{v}_{2}^{h})_{E^{k}}
\leq ||Q_{2}^{h}T^{k}_{0}\bm{v}_{2}^{h}||_{E^{k}}||\bm{v}_{2}^{h}||_{E^{k}}.
\end{equation}
Meanwhile,
\begin{align}
||Q_{2}^{h}T^{k}_{0}\bm{v}_{2}^{h}||^{2}_{E^{k}}&=||T^{k}_{0}\bm{v}_{2}^{h}||^{2}_{E^{k}}+
(\lambda^{k}-\lambda_{1}^{h})||Q_{1}^{h}T^{k}_{0}\bm{v}_{2}^{h}||^{2}_{b}+\lambda^{k}||K_{0}^{h}T^{k}_{0}\bm{v}_{2}^{h}||^{2}_{b}\notag\\
&\leq ||T^{k}_{0}\bm{v}_{2}^{h}||^{2}_{E^{k}}+(\lambda^{k}-\lambda_{1}^{h})||T^{k}_{0}\bm{v}_{2}^{h}||^{2}_{b}+
\lambda^{k}||T^{k}_{0}\bm{v}_{2}^{h}||^{2}_{b}\notag\\
&\leq \{1+C\lambda_{1}+CH^{2}\} ||T^{k}_{0}\bm{v}_{2}^{h}||^{2}_{E^{k}}\label{T0k2}.
\end{align}
Combining $\eqref{Tok1}$ and $\eqref{T0k2}$ together, we have
\begin{equation}\label{Gkcoarse}
||T^{k}_{0}\bm{v}_{2}^{h}||^{2}_{E^{k}}\leq (1+C\lambda_{1}+CH^{2}) ||\bm{v}_{2}^{h}||^{2}_{E^{k}}.
\end{equation}
For the third term of $\eqref{Gkpri}$, we have
\begin{equation}\label{Tik1}
\sum_{i=1}^{N}||T_{i}^{k}\bm{v}_{2}^{h}||^{2}_{E^{k}}=\sum_{i=1}^{N}(T_{i}^{k}\bm{v}_{2}^{h},\bm{v}_{2}^{h})_{E^{k}}
\leq ||Q_{2}^{h}\sum_{i=1}^{N}T_{i}^{k}\bm{v}_{2}^{h}||_{E^{k}}||\bm{v}_{2}^{h}||_{E^{k}}.
\end{equation}
Owing to $\eqref{discretespectralspaces}$, Lemma $\ref{strengthenedCauchySchwarzinequality}$, Lemma $\ref{normequivalence}$ and the Poincar$\acute{e}$ inequality, we get
\begin{align}
&\ \ \ \ ||Q_{2}^{h}\sum_{i=1}^{N}T_{i}^{k}\bm{v}_{2}^{h}||^{2}_{E^{k}}\leq C||Q_{2}^{h}\sum_{i=1}^{N}T_{i}^{k}\bm{v}_{2}^{h}||^{2}_{E^{h}}
\leq C||\sum_{i=1}^{N}T_{i}^{k}\bm{v}_{2}^{h}||^{2}_{a}+
C\lambda_{1}||\sum_{i=1}^{N}T_{i}^{k}\bm{v}_{2}^{h}||^{2}_{b}\notag\\
&=C\sum_{i=1}^{N}\sum_{l=1}^{N}a(T_{i}^{k}\bm{v}_{2}^{h},T_{l}^{k}\bm{v}_{2}^{h})+
C\lambda_{1}\sum_{i=1}^{N}\sum_{l=1}^{N}b(T_{i}^{k}\bm{v}_{2}^{h},T_{l}^{k}\bm{v}_{2}^{h})\notag\\
&\leq CN_{0}\sum_{i=1}^{N}||T_{i}^{k}\bm{v}_{2}^{h}||_{a}^{2}+CN_{0}\lambda_{1}\sum_{i=1}^{N}||T_{i}^{k}\bm{v}_{2}^{h}||_{b}^{2}\leq CN_{0}(1+\lambda_{1}H^{2})\sum_{i=1}^{N}||T_{i}^{k}\bm{v}_{2}^{h}||_{E^{k}}^{2}\label{Tik2}.
\end{align}
Combining $\eqref{Tik1}$ and $\eqref{Tik2}$ together, we get
\begin{equation}\label{Gkifinelocal}
\sum_{i=1}^{N}||T_{i}^{k}\bm{v}_{2}^{h}||^{2}_{E^{k}}\leq CN_{0}(1+\lambda_{1}H^{2})||\bm{v}_{2}^{h}||^{2}_{E^{k}},
\end{equation}
which, together with $\eqref{Gkpri}$ $\eqref{Gkcoarse}$, yields
\begin{align*}
(G^{k}\bm{v}_{2}^{h},\bm{v}_{2}^{h})_{E^{k}}&=||\bm{v}_{2}^{h}||^{2}_{E^{k}}-\alpha(T^{k}_{0}\bm{v}_{2}^{h},\bm{v}_{2}^{h})_{E^{k}}-
\alpha\sum_{i=1}^{N}||T_{i}^{k}\bm{v}_{2}^{h}||^{2}_{E^{k}}\\
&\geq \{1-\alpha(1+CN_{0}+C\lambda_{1}+CH^{2})\}||\bm{v}_{2}^{h}||^{2}_{E^{k}},
\end{align*}
where $ C=O(1),\ \lambda_{1}=O(1),\ N_{0}=O(1)$. If we take $0<\alpha<\alpha_{0}=\frac{1}{1+CN_{0}+C\lambda_{1}+CH^{2}}$, then we may obtain the conclusion of this lemma.  \qed
\end{proof}
\begin{lemma}\label{Gkstabledecomposition}
For the error space $M_{h}^{\perp}(\lambda_{1})$, it holds that
\[M_{h}^{\perp}(\lambda_{1})=Q_{2}^{h}M_{H}^{\perp}(\lambda_{1})+\sum_{i=1}^{N}Q_{2}^{h}V_{0}^{(i)}\]
and there exist $\bm{w}_{0}\in Q_{2}^{h}M_{H}^{\perp}(\lambda_{1})$ and $\bm{w}_{0}^{(i)}\in V_{0}^{(i)}$ such that
\begin{equation}\label{Decomposition}
(\bm{w}_{0},\bm{w}_{0})_{E^{k}}+\sum_{i=1}^{N}(\bm{w}_{0}^{(i)},\bm{w}_{0}^{(i)})_{E^{k}}\leq CN_{0}(1+\frac{H^{2}}{\delta^{2}})(\bm{v}_{2}^{h},\bm{v}_{2}^{h})_{E^{k}}.
\end{equation}
\end{lemma}
\begin{proof}
For any $\bm{v}_{2}^{h}\in M_{h}^{\perp}(\lambda_{1})$, we take $\bm{w}_{0}=Q_{2}^{h}Q_{2}^{H}Z^{(0)}Q^{H}P^{h}\bm{v}_{2}^{h}$ and $ \bm{w}_{0}^{(i)}=Z^{(i)}r_{h}\theta_{i}(\bm{v}_{2}^{h}-\bm{w}_{0})$, where $r_{h}$ is the standard edge element interpolation operator. Then
\begin{align}
\bm{w}_{0}+\sum_{i=1}^{N}Q_{2}^{h}\bm{w}_{0}^{(i)}&=\bm{w}_{0}+\sum_{i=1}^{N}Q_{2}^{h}Z^{(i)}r_{h}\theta_{i}(\bm{v}_{2}^{h}-\bm{w}_{0})\notag\\
&=\bm{w}_{0}+\sum_{i=1}^{N}Q_{2}^{h}(I-K^{(i)})r_{h}\theta_{i}(\bm{v}_{2}^{h}-\bm{w}_{0})\label{decomposition1},
\end{align}
which, together with the fact that for any $\bm{v}_{h}\in E_{h}(\Omega_{h}),\ K^{(i)}\bm{v}_{h}\in \nabla{S_{h}(\Omega_{h})}$, yields
\begin{align*}
&\ \ \ \ \bm{w}_{0}+\sum_{i=1}^{N}Q_{2}^{h}\bm{w}_{0}^{(i)}=\bm{w}_{0}+\sum_{i=1}^{N}Q_{2}^{h}r_{h}\theta_{i}(\bm{v}_{2}^{h}-\bm{w}_{0})\\
&=\bm{w}_{0}+Q_{2}^{h}r_{h}(\sum_{i=1}^{N}\theta_{i})(\bm{v}_{2}^{h}-\bm{w}_{0})=\bm{w}_{0}+Q_{2}^{h}r_{h}(\bm{v}_{2}^{h}-\bm{w}_{0})=\bm{v}_{2}^{h}.
\end{align*}
\par Next, we prove $\eqref{Decomposition}$. For the component in the coarse space, owing to Lemma $\ref{normequivalence}$, the orthogonality on $E_{h}(\Omega_{h})$ and $E_{H}(\Omega_{H})$ in the sense of $a(\cdot,\cdot)$ and $b(\cdot,\cdot)$, the definition of $Z^{(0)}$ and the Poincar$\acute{e}$ inequality, we get
\begin{align}
||\bm{w}_{0}||^{2}_{E^{k}}&\leq C||Q_{2}^{h}Q_{2}^{H}Z^{(0)}Q^{H}P^{h}\bm{v}_{2}^{h}||^{2}_{E^{h}}\notag\\
&\leq C\{||Q_{2}^{H}Z^{(0)}Q^{H}P^{h}\bm{v}_{2}^{h}||^{2}_{E^{h}}+
\lambda_{1}^{h}||K_{0}^{h}Q_{2}^{H}Z^{(0)}Q^{H}P^{h}\bm{v}_{2}^{h}||^{2}_{b}\}\notag\\
&\leq C||Z^{(0)}Q^{H}P^{h}\bm{v}_{2}^{h}||^{2}_{a}+C||Z^{(0)}Q^{H}P^{h}\bm{v}_{2}^{h}||_{b}^{2}
+C\lambda_{1}||K_{0}^{h}Q_{2}^{H}Z^{(0)}Q^{H}P^{h}\bm{v}_{2}^{h}||^{2}_{b}\notag\\
&\leq C||Z^{(0)}Q^{H}P^{h}\bm{v}_{2}^{h}||^{2}_{a}+C(1+\lambda_{1})||Z^{(0)}Q^{H}P^{h}\bm{v}_{2}^{h}||_{b}^{2}\notag\\
&\leq C||Z^{(0)}Q^{H}P^{h}\bm{v}_{2}^{h}||^{2}_{a}\leq C||\bm{curl}Q^{H}P^{h}\bm{v}_{2}^{h}||^{2}_{b}\label{Coarsecomponentestimate}.
\end{align}
 Using Lemma $\ref{QH}$, the definition of $P^{h}$ and the embedding theorem $\bm{H}_{0}(\bm{curl};\Omega)\cap \bm{H}(div_{0};\Omega;\epsilon_{r})\hookrightarrow H^{1}(\Omega)^{3}$, we get
\begin{equation}\label{QHPh}
||\bm{curl}Q^{H}P^{h}\bm{v}_{2}^{h}||^{2}_{b}\leq C|P^{h}\bm{v}_{2}^{h}|^{2}_{1}\leq C||\bm{curl} P^{h}\bm{v}_{2}^{h}||^{2}_{0}
\leq C||\bm{curl}\ \bm{v}_{2}^{h}||^{2}_{0}\leq C||\bm{v}_{2}^{h}||^{2}_{a}\leq C||\bm{v}_{2}^{h}||^{2}_{E^{k}}.
\end{equation}
Then
\begin{equation}\label{coarsedecom}
||\bm{w}_{0}||^{2}_{E^{k}}\leq C ||\bm{v}_{2}^{h}||^{2}_{E^{k}}.
\end{equation}
\par For the local fine component, we may use the properties of the  partition of unity $\eqref{unitypartition}$ and the property of the interpolation operator $r_{h}$ to prove that
\begin{align}
&\ \ \ \ \ \sum_{i=1}^{N}||Z^{(i)}r_{h}\theta_{i}(\bm{v}_{2}^{h}-\bm{w}_{0})||^{2}_{a}
\leq C\sum_{i=1}^{N}||\bm{curl}\ r_{h} \theta_{i}(\bm{v}_{2}^{h}-\bm{w}_{0})||^{2}_{b,\Omega^{'}_{i}}\notag\\
&\leq CN_{0}\{\frac{1}{\delta^{2}}||\bm{v}_{2}^{h}-\bm{w}_{0}||^{2}_{b}+ ||\bm{curl}\ (\bm{v}_{2}^{h}-\bm{w}_{0})||^{2}_{b}\},\label{localcomponent}
\end{align}
where the last inequality holds (see \cite{Toselli} or Lemma 10.9 in \cite{ToselliM}).
We first estimate the first term of $\eqref{localcomponent}$. Combining the triangle inequality, $\eqref{QHPh}$, Lemma $\ref{lemmaPh}$ and Lemma $\ref{QH}$, we have
\begin{align}
||\bm{v}_{2}^{h}-\bm{w}_{0}||^{2}_{b}&=||\bm{v}_{2}^{h}-Q_{2}^{h}Q_{2}^{H}Z^{(0)}Q^{H}P^{h}\bm{v}_{2}^{h}||^{2}_{b}
\leq ||\bm{v}_{2}^{h}-Q_{2}^{H}Z^{(0)}Q^{H}P^{h}\bm{v}_{2}^{h}||^{2}_{b}\notag\\
&\leq C\{||\bm{v}_{2}^{h}-P^{h}\bm{v}_{2}^{h}||_{b}^{2}+||P^{h}\bm{v}_{2}^{h}-Q^{H}P^{h}\bm{v}_{2}^{h}||_{b}^{2}
\notag\\&\ \ \ +||Q^{H}P^{h}\bm{v}_{2}^{h}-Z^{(0)}Q^{H}P^{h}\bm{v}_{2}^{h}||^{2}_{b}+||Z^{(0)}Q^{H}P^{h}\bm{v}_{2}^{h}-Q_{2}^{H}Z^{(0)}Q^{H}P^{h}\bm{v}_{2}^{h}||^{2}_{b}\}\notag\\
&\leq  C\{H^{2}||\bm{curl}\ \bm{v}_{2}^{h}||_{b}^{2}+||Q^{H}P^{h}\bm{v}_{2}^{h}-Z^{(0)}Q^{H}P^{h}\bm{v}_{2}^{h}||^{2}_{b}\notag\\
&\ \ \ +||Z^{(0)}Q^{H}P^{h}\bm{v}_{2}^{h}-Q_{2}^{H}Z^{(0)}Q^{H}P^{h}\bm{v}_{2}^{h}||^{2}_{b}\}\label{localcomponentfirst1}.
\end{align}
On the one hand, it is easy to see that $Q^{H}P^{h}\bm{v}_{2}^{h}-Z^{(0)}Q^{H}P^{h}\bm{v}_{2}^{h}=K_{0}^{H}Q^{H}P^{h}\bm{v}_{2}^{h}$. So by using the orthogonality of $\bm{v}_{2}^{h}\perp_{b(\cdot,\cdot)} K_{0}^{H}Q^{H}P^{h}\bm{v}_{2}^{h}$, we know that the second term  of $\eqref{localcomponentfirst1}$ may be estimate as follows:
\begin{align*}
&\ \ \ \ ||K_{0}^{H}Q^{H}P^{h}\bm{v}_{2}^{h}||^{2}_{b}=b(K_{0}^{H}Q^{H}P^{h}\bm{v}_{2}^{h},Q^{H}P^{h}\bm{v}_{2}^{h})\\
&=b(K_{0}^{H}Q^{H}P^{h}\bm{v}_{2}^{h},Q^{H}P^{h}\bm{v}_{2}^{h}-\bm{v}_{2}^{h})
\leq ||K_{0}^{H}Q^{H}P^{h}\bm{v}_{2}^{h}||_{b}||Q^{H}P^{h}\bm{v}_{2}^{h}-\bm{v}_{2}^{h}||_{b}.
\end{align*}
Then, combining Lemma $\ref{lemmaPh}$ and Lemma $\ref{QH}$, we obtain
\[||K_{0}^{H}Q^{H}P^{h}\bm{v}_{2}^{h}||^{2}_{b}\leq ||Q^{H}P^{h}\bm{v}_{2}^{h}-\bm{v}_{2}^{h}||^{2}_{b}\leq CH^{2}||\bm{curl}\ \bm{v}_{2}^{h}||^{2}_{b},\]
that is
\begin{equation}\label{localcomponentfirst2}
||Q^{H}P^{h}\bm{v}_{2}^{h}-Z^{(0)}Q^{H}P^{h}\bm{v}_{2}^{h}||_{b}^{2}\leq CH^{2}||\bm{curl}\ \bm{v}_{2}^{h}||^{2}_{b}.
\end{equation}
On the other hand, we know $Z^{(0)}Q^{H}P^{h}\bm{v}_{2}^{h}-Q_{2}^{H}Z^{(0)}Q^{H}P^{h}\bm{v}_{2}^{h}=Q_{1}^{H}Z^{(0)}Q^{H}P^{h}\bm{v}_{2}^{h}.$
 Using the same argument with the proof of Lemma $\ref{LemmaK0hQ2H}$, we know that there exists $\bm{w}_{1}^{h}\in M_{h}(\lambda_{1})$ such that
 \[||Q_{1}^{H}Z^{(0)}Q^{H}P^{h}\bm{v}_{2}^{h}-\bm{w}_{1}^{h}||_{b}\leq CH||Q_{1}^{H}Z^{(0)}Q^{H}P^{h}\bm{v}_{2}^{h}||_{b}.\]
 Hence, by the orthogonality of $\bm{w}_{1}^{h}\perp_{b(\cdot,\cdot)} \bm{v}_{2}^{h}$, Cauchy-Schwarz inequality, Lemma $\ref{lemmaPh}$ and Lemma $\ref{QH}$, we obtain
 \begin{align*}
&\ \ \ \ ||Q_{1}^{H}Z^{(0)}Q^{H}P^{h}\bm{v}_{2}^{h}||^{2}_{b}\\
&=b(Q_{1}^{H}Z^{(0)}Q^{H}P^{h}\bm{v}_{2}^{h},Q^{H}P^{h}\bm{v}_{2}^{h}-\bm{v}_{2}^{h})
+b(Q_{1}^{H}Z^{(0)}Q^{H}P^{h}\bm{v}_{2}^{h},\bm{v}_{2}^{h})\\
&=b(Q_{1}^{H}Z^{(0)}Q^{H}P^{h}\bm{v}_{2}^{h},Q^{H}P^{h}\bm{v}_{2}^{h}-\bm{v}_{2}^{h})+
b(Q_{1}^{H}Z^{(0)}Q^{H}P^{h}\bm{v}_{2}^{h}-\bm{w}_{1}^{h},\bm{v}_{2}^{h})\\
&\leq||Q_{1}^{H}Z^{(0)}Q^{H}P^{h}\bm{v}_{2}^{h}||_{b}||Q^{H}P^{h}\bm{v}_{2}^{h}-\bm{v}_{2}^{h}||_{b}+
||Q_{1}^{H}Z^{(0)}Q^{H}P^{h}\bm{v}_{2}^{h}-\bm{w}_{1}^{h}||_{b}||\bm{v}_{2}^{h}||_{b}\\
&\leq CH||Q_{1}^{H}Z^{(0)}Q^{H}P^{h}\bm{v}_{2}^{h}||_{b}||\bm{curl}\bm{v}_{2}^{h}||_{b}+
CH||Q_{1}^{H}Z^{(0)}Q^{H}P^{h}\bm{v}_{2}^{h}||_{b}||\bm{curl}\bm{v}_{2}^{h}||_{b}.
\end{align*}
Then, by the definition of $Q_{1}^{h},\ Z^{(0)}$, we get
\begin{equation}\label{localcomponentfirst3}
 ||Z^{(0)}Q^{H}P^{h}\bm{v}_{2}^{h}-Q_{2}^{H}Z^{(0)}Q^{H}P^{h}\bm{v}_{2}^{h}||_{b}^{2}\leq  ||Q_{1}^{H}Z^{(0)}Q^{H}P^{h}\bm{v}_{2}^{h}||^{2}_{b}\leq CH^{2}||\bm{curl}\bm{v}_{2}^{h}||_{b}^{2},
\end{equation}
which, together with $\eqref{localcomponentfirst1}$, $\eqref{localcomponentfirst2}$, yields
\begin{align}
||\bm{v}_{2}^{h}-\bm{w}_{0}||^{2}_{b}
&\leq C\{H^{2}||\bm{curl}\bm{v}_{2}^{h}||_{b}^{2}+||K_{0}^{H}Q^{H}P^{h}\bm{v}_{2}^{h}||^{2}_{b}
+||Q_{1}^{H}Z^{0}Q^{H}P^{h}\bm{v}_{2}^{h}||^{2}_{b}\}\notag\\
&\leq CH^{2}||\bm{curl}\bm{v}_{2}^{h}||_{b}^{2}
\leq CH^{2}||\bm{v}_{2}^{h}||_{a}^{2}\label{localcomponentfirst4}.
\end{align}
For the second term of $\eqref{localcomponent}$, by $\eqref{coarsedecom}$ and Lemma $\ref{normequivalence}$, we have
\begin{align}
&\ \ \ \ ||\bm{curl}(\bm{v}_{2}^{h}-\bm{w}_{0})||^{2}_{b}
\leq  C\{||\bm{curl}\bm{v}_{2}^{h}||^{2}_{b}+||\bm{curl}\bm{w}_{0}||_{b}^{2}\}\notag\\
&\leq  C\{||\bm{curl}\bm{v}_{2}^{h}||^{2}_{b}+||\bm{w}_{0}||_{a}^{2} \}\leq   C\{||\bm{curl}\bm{v}_{2}^{h}||^{2}_{b}+||\bm{w}_{0}||_{E^{k}}^{2} \}
\leq  C||\bm{v}_{2}^{h}||_{a}^{2}.\label{localcomponentsecond}
\end{align}
Combining $\eqref{localcomponentsecond}$ and $\eqref{localcomponentfirst4}$, we know
\begin{equation}\label{Finecomponentestimate}
\sum_{i=1}^{N}||Z^{(i)}r_{h}\theta_{i}(\bm{v}_{2}^{h}-\bm{w}_{0})||^{2}_{a}
\leq CN_{0}(1+\frac{H^{2}}{\delta^{2}})||\bm{v}_{2}^{h}||_{a}^{2}.
\end{equation}
Finally, by Lemma $\ref{normequivalence}$, $\eqref{Coarsecomponentestimate}$ and $\eqref{Finecomponentestimate}$, we may obtain the proof of $\eqref{Decomposition}$.     \qed
\end{proof}

\par\noindent{\bf Proof of Theorem $\ref{Gkv2theorem}$}:\ \
 Because of Lemma $\ref{Gkstabledecomposition}$, we know that for any $\bm{v}_{2}^{h}\in M_{h}^{\perp}(\lambda_{1})$, there exist $\bm{w}_{0}\in Q_{2}^{h}M_{H}^{\perp}(\lambda_{1})$ and $\bm{w}_{i}\in V_{0}^{(i)}$ such that
\begin{equation}\label{stabledecomposition}
\bm{v}_{2}^{h}=\bm{w}_{0}+\sum_{i=1}^{N}Q_{2}^{h}\bm{w}_{i}\ \ \ \text{and}\ \ \sum_{i=0}^{N}||\bm{w}_{i}||^{2}_{E^{k}}\leq CN_{0}(1+\frac{H^{2}}{\delta^{2}})||\bm{v}_{2}^{h}||_{E^{k}}^{2}.
\end{equation}
By $\eqref{T0k},\ \eqref{Tik}$, $\eqref{stabledecomposition}$ and  Cauchy-Schwarz inequality, we may obtain
\begin{equation}
(\bm{v}_{2}^{h},\bm{v}_{2}^{h})_{E^{k}}\leq CN_{0}(1+\frac{H^{2}}{\delta^{2}})(Q_{2}^{h}\sum_{i=0}^{N}T_{i}^{k}\bm{v}_{2}^{h},\bm{v}_{2}^{h})_{E^{k}}.
\end{equation}
Moreover,
\begin{align*}
(G^{k}\bm{v}_{2}^{h},\bm{v}_{2}^{h})_{E^{k}}&=(\bm{v}_{2}^{h},\bm{v}_{2}^{h})_{E^{k}}-
\alpha(Q_{2}^{h}\sum_{i=0}^{N}T_{i}^{k}\bm{v}_{2}^{h},\bm{v}_{2}^{h})_{E^{k}}\\
&\leq (1-C\frac{1}{1+\frac{H^{2}}{\delta^{2}}})(\bm{v}_{2}^{h},\bm{v}_{2}^{h})_{E^{k}}\leq (1-C\frac{\delta^{2}}{H^{2}})(\bm{v}_{2}^{h},\bm{v}_{2}^{h})_{E^{k}}.
\end{align*}
  Based on Lemma $\ref{Gksympos}$, we may define $(G^{k})^{1/2}:M_{h}^{\perp}(\lambda_{1})\to M_{h}^{\perp}(\lambda_{1})$ and know that $(G^{k})^{1/2}$ is also a symmetric positive definite operator on $(\cdot,\cdot)_{E^{k}}$. Hence, we have
\begin{align*}
&\ \ \ \ (G^{k}\bm{v}_{2}^{h},G^{k}\bm{v}_{2}^{h})_{E^{k}}=(G^{k}(G^{k})^{1/2}\bm{v}_{2}^{h},(G^{k})^{1/2}\bm{v}_{2}^{h})_{E^{k}}\\
&\leq (1-C\frac{\delta^{2}}{H^{2}})((G^{k})^{1/2}\bm{v}_{2}^{h},(G^{k})^{1/2}\bm{v}_{2}^{h})_{E^{k}}
=(1-C\frac{\delta^{2}}{H^{2}})(G^{k}\bm{v}_{2}^{h},\bm{v}_{2}^{h})_{E^{k}}\\&\leq (1-C\frac{\delta^{2}}{H^{2}})^{2}(\bm{v}_{2}^{h},\bm{v}_{2}^{h})_{E^{k}}.
\end{align*}
Then
\begin{equation}\label{Gkestimate}
||G^{k}\bm{v}_{2}^{h}||_{E^{k}}\leq (1-C\frac{\delta^{2}}{H^{2}})||\bm{v}_{2}^{h}||_{E^{k}},
\end{equation}
which completes the proof of this theorem. \qed
\subsection{Estimate of the almost counterbalanced term $I_{2}$}
\par Before we analyze the estimate of the almost counterbalanced term $I_{2}$, we first estimate the orthogonal parameter $\beta^{k}=-\frac{b((B_{h}^{k})^{-1}\bm{r}^{k},\bm{u}^{k})}{b((B_{h}^{k})^{-1}\bm{u}^{k},\bm{u}^{k})}$.
\begin{lemma}\label{lemmaorthogonalparameter}
For sufficiently small $H$, it holds that
\begin{equation}\notag
|\beta^{k}|
\leq CH\sqrt{\lambda^{k}-\lambda_{1}^{h}},
\end{equation}
where $C$ is independent of $H,\ h$.
\end{lemma}
\begin{proof} First, we decompose the numerator of $|\beta^{k}|$. Due to the definition of $\bm{r}^{k}, (B_{h}^{k})^{-1}$ and the fact $\bm{u}^{k}=\bm{u}_{1}^{k}-\bm{e}_{2}^{k}\in E_{h}^{0}(\Omega_{h};\epsilon_{r})$, we get
\begin{align}
&\ \ \ \ |b((B_{h}^{k})^{-1}\bm{r}^{k},\bm{u}^{k})|\notag\\
&=|b((B_{h}^{k})^{-1}(\lambda^{k}-A^{h})\bm{u}_{1}^{k},\bm{u}^{k})+b((B_{h}^{k})^{-1}(A^{h}-\lambda^{k})\bm{e}_{2}^{k},\bm{u}^{k})|\notag\\
&\leq(\lambda^{k}-\lambda_{1}^{h})|b((B_{h}^{k})^{-1}\bm{u}_{1}^{k},\bm{u}^{k})|+|b((B_{h}^{k})^{-1}(A^{h}-\lambda^{k})\bm{e}_{2}^{k},\bm{u}^{k})|\notag\\
&=(\lambda^{k}-\lambda_{1}^{h})|b((B_{h}^{k})^{-1}\bm{u}_{1}^{k},\bm{u}^{k})|+|b((B_{0}^{k})^{-1}Q_{1}^{H}Z^{(0)}Q^{H}(A^{h}-\lambda^{k})\bm{e}_{2}^{k},\bm{u}^{k})|\notag\\
&\ \ \ +|b((B_{0}^{k})^{-1}Q_{2}^{H}Z^{(0)}Q^{H}(A^{h}-\lambda^{k})\bm{e}_{2}^{k},\bm{u}^{k})+
|b(\sum_{i=1}^{N}(B_{i}^{k})^{-1}Z^{(i)}Q^{(i)}(A^{h}-\lambda^{k})\bm{e}_{2}^{k},\bm{u}^{k})|\notag\\
&\leq(\lambda^{k}-\lambda_{1}^{h})|b((B_{h}^{k})^{-1}\bm{u}_{1}^{k},\bm{u}^{k})|+|b((B_{0}^{k})^{-1}Q_{1}^{H}Z^{(0)}Q^{H}(A^{h}-\lambda^{k})\bm{e}_{2}^{k},\bm{u}^{k})|    \notag \\&\ \ \ +||T_{0}^{k}\bm{e}_{2}^{k}||_{b}+||\sum_{i=1}^{N}T_{i}^{k}\bm{e}_{2}^{k}||_{b}:=(\lambda^{k}-\lambda_{1}^{h})\Gamma_{1}+\Gamma_{2}+\Gamma_{3}+\Gamma_{4},\label{numerator}
\end{align}
where the last inequality holds owing to Cauchy-Schwarz inequality and $||\bm{u}^{k}||_{b}=1$. In fact, for $k=1$, it is obvious that $||\bm{u}^{1}||_{b}=1$ in algorithm 2. As $\bm{u}^{k}$ is the minimum value point of Rayleigh quotient in $W^{k} (k\geq 2)$ in algorithm 2, without loss of generality, let $||\bm{u}^{k}||_{b}=1$.
Next, we estimate  $\eqref{numerator}$ one by one. For the first term $\Gamma_{1}$ in $\eqref{numerator}$, we have
\begin{align}
&\ \ \ \ |b((B_{h}^{k})^{-1}\bm{u}_{1}^{k},\bm{u}^{k})|\notag\\
&=|b((B_{0}^{k})^{-1}Q_{1}^{H}Z^{(0)}Q^{H}\bm{u}_{1}^{k},\bm{u}^{k})|+
|b((B_{0}^{k})^{-1}Q_{2}^{H}Z^{(0)}Q^{H}\bm{u}_{1}^{k},\bm{u}^{k})|+
|\sum_{i=1}^{N}b((B_{i}^{k})^{-1}Z^{(i)}Q^{(i)}\bm{u}_{1}^{k},\bm{u}^{k})|\notag\\
&\leq \{||(B_{0}^{k})^{-1}Q_{1}^{H}Z^{(0)}Q^{H}\bm{u}_{1}^{k}||_{b}+
||(B_{0}^{k})^{-1}Q_{2}^{H}Z^{(0)}Q^{H}\bm{u}_{1}^{k}||_{b}
+\sum_{i=1}^{N}||(B_{i}^{k})^{-1}Z^{(i)}Q^{(i)}\bm{u}_{1}^{k}||_{b}\}||\bm{u}^{k}||_{b}.\label{numerator1}
\end{align}
By the definition of $B_{0}^{k}$, $\eqref{Bik}$, $\eqref{Q2HQ1}$ and Assumption $1$, we know that for sufficiently small $H$,
\begin{align}
&\ \ \ \ |b((B_{h}^{k})^{-1}\bm{u}_{1}^{k},\bm{u}^{k})| \leq \frac{1}{|\lambda^{k}-\lambda_{1}^{H}|}||Q_{1}^{H}\bm{u}_{1}^{k}||_{b}+CH||\bm{u}_{1}^{k}||_{b}+
CH^{2}\sum_{i=1}^{N}||\bm{u}_{1}^{k}||_{b,\Omega^{'}_{i}}\notag\\
&\leq \frac{C}{|\lambda^{k}-\lambda_{1}^{H}|}||\bm{u}_{1}^{k}||_{b}+CH||\bm{u}_{1}^{k}||_{b}+
CN_{0}H^{2}||\bm{u}_{1}^{k}||_{b}\leq\frac{C}{|\lambda^{k}-\lambda_{1}^{H}|}.\label{numeratorfirst}
\end{align}
For the second term $\Gamma_{2}$ in $\eqref{numerator}$, by the definition of $Q_{1}^{H},\ Q_{2}^{h}$, Lemma $\ref{LemmaK0hQ2H}$ and $\eqref{e2Eku1b}$, we get
\begin{align}
&\ \ \ \ |b((B_{0}^{k})^{-1}Q_{1}^{H}Z^{(0)}Q^{H}(A^{h}-\lambda^{k})\bm{e}_{2}^{k},\bm{u}^{k})|
=\frac{1}{|\lambda^{k}-\lambda_{1}^{H}|}|b(Q_{1}^{H}(A^{h}-\lambda^{k})\bm{e}_{2}^{k},\bm{u}^{k})|\notag\\
&=\frac{1}{|\lambda^{k}-\lambda_{1}^{H}|}|(\bm{e}_{2}^{k},Q_{2}^{h}Q_{1}^{H}\bm{u}^{k})_{E^{k}}|\leq \frac{C}{|\lambda^{k}-\lambda_{1}^{H}|}||\bm{e}_{2}^{k}||_{E^{k}}||Q_{2}^{h}Q_{1}^{H}\bm{u}^{k}||_{a}\leq \frac{CH\sqrt{\lambda^{k}-\lambda_{1}^{h}}}{|\lambda^{k}-\lambda_{1}^{H}|}.\label{numeratorsecond}
\end{align}
For the third term $\Gamma_{3}$ in $\eqref{numerator}$, by the Poincar$\acute{e}$ inequality, $\eqref{e2Eku1b}$ and $\eqref{Gkcoarse}$, we obtain
\begin{equation}\label{numeratorthird}
||T_{0}^{k}\bm{e}_{2}^{k}||_{b}\leq C||T_{0}^{k}\bm{e}_{2}^{k}||_{E^{k}}\leq C||\bm{e}_{2}^{k}||_{E^{k}}
=C\sqrt{\lambda^{k}-\lambda_{1}^{h}}||\bm{u}_{1}^{k}||_{b}\leq C\sqrt{\lambda^{k}-\lambda_{1}^{h}}.
\end{equation}
For the fourth term $\Gamma_{4}$ in $\eqref{numerator}$, by Lemma $\ref{strengthenedCauchySchwarzinequality}$, the Poincar$\acute{e}$ inequality, $\eqref{e2Eku1b}$ and $\eqref{Gkifinelocal}$, we get
\begin{equation}\label{numeratorfourth}
||\sum_{i=1}^{N}T_{i}^{k}\bm{e}_{2}^{k}||_{b}
\leq \sqrt{N_{0}\sum_{i=1}^{N}||T_{i}^{k}\bm{e}_{2}^{k}||^{2}_{b}}
\leq CH\sqrt{N_{0}\sum_{i=1}^{N}||T_{i}^{k}\bm{e}_{2}^{k}||^{2}_{E^{k}}}
\leq CH||\bm{e}_{2}^{k}||_{E^{k}}
\leq CH\sqrt{\lambda^{k}-\lambda_{1}^{h}}.
\end{equation}
Combining $\eqref{numerator}$ and  $\eqref{numeratorfirst}\sim\eqref{numeratorfourth}$ together, for sufficiently small $H$, we have
\begin{align}\label{estimatnum}
|b((B_{h}^{k})^{-1}\bm{r}^{k},\bm{u}^{k})|\leq \frac{\lambda^{k}-\lambda_{1}^{h}+CH\sqrt{\lambda^{k}-\lambda_{1}^{h}}}{|\lambda^{k}-\lambda_{1}^{H}|}.
\end{align}
\par On the other hand, we analyze the denominator of $|\beta^{k}|$
\begin{align}
|b((B_{h}^{k})^{-1}\bm{u}^{k},\bm{u}^{k})|&\geq|b((B_{0}^{k})^{-1}Q_{1}^{H}Q^{H}\bm{u}^{k},\bm{u}^{k})|
-|b((B_{0}^{k})^{-1}Q_{2}^{H}Q^{H}\bm{u}^{k},\bm{u}^{k})|\notag\\
&\ \ \
-|b(\sum_{i=1}^{N}(B_{i}^{k})^{-1}Z^{(i)}Q^{(i)}\bm{u}^{k},\bm{u}^{k})|
:=J_{1}-J_{2}-J_{3}.\label{estimatedenoest}
\end{align}
Similarly, we estimate the terms in $\eqref{estimatedenoest}$ one by one. Note that by Lemma $\ref{LemmaK0hQ2H}$, we have
\begin{align}
||Q_{2}^{H}\bm{u}^{k}||_{b} &\leq ||Q_{2}^{H}\bm{u}_{1}^{k}||_{b}+||Q_{2}^{H}\bm{e}_{2}^{k}||_{b}\leq CH||\bm{u}_{1}^{k}||_{b}+C||\bm{e}_{2}^{k}||_{E^{k}}\notag\\
&\leq (CH+C\sqrt{\lambda^{k}-\lambda_{1}^{h}} )||\bm{u}_{1}^{k}||_{b}\leq CH,\label{K0HQ1HQ2Hfirst}
\end{align}
and
\begin{equation}\label{K0HQ1HQ2Hsecond}
||K_{0}^{H}\bm{u}^{k}||_{b} \leq ||K_{0}^{H}\bm{u}_{1}^{k}||_{b}+||K_{0}^{H}\bm{e}_{2}^{k}||_{b}\leq CH||\bm{u}_{1}^{k}||_{b}+C||\bm{e}_{2}^{k}||_{E^{k}}\leq CH.
\end{equation}
Meanwhile,
\begin{align}
||Q_{1}^{H}\bm{u}^{k}||^{2}_{b}&=||Q^{H}\bm{u}^{k}||^{2}_{b}-||Q_{2}^{H}\bm{u}^{k}||^{2}_{b}-||K_{0}^{H}\bm{u}^{k}||^{2}_{b}\notag\\
&\geq ||\bm{u}^{k}||_{b}^{2}-||\bm{u}^{k}-Q^{H}\bm{u}^{k}||_{b}^{2}-CH^{2}\geq 1-CH^{2},\label{K0HQ1HQ2Hthird}
\end{align}
where the last inequality holds due to the following estimate
\begin{align*}
||\bm{u}^{k}-Q^{H}\bm{u}^{k}||_{b}&\leq ||\bm{u}^{k}-Q^{H}P^{h}\bm{u}^{k}||_{b}+||Q^{H}P^{h}\bm{u}^{k}-Q^{H}\bm{u}^{k}||_{b}\\
&\leq C||\bm{u}^{k}-P^{h}\bm{u}^{k}||_{b}+||P^{h}\bm{u}^{k}-Q^{H}P^{h}\bm{u}^{k}||_{b}\\
&\leq Ch||\bm{curl}\bm{u}^{k}||_{b}+CH|P^{h}\bm{u}^{k}|_{1}\leq Ch||\bm{curl}\bm{u}^{k}||_{b}+CH||\bm{curl}P^{h}\bm{u}^{k}||_{b}\\
&\leq CH||\bm{curl}\bm{u}^{k}||_{b}\leq CH||\bm{u}^{k}||_{a}=CH\sqrt{\lambda^{k}}||\bm{u}^{k}||_{b}\leq CH,
\end{align*}
where the similar argument in $\eqref{QHPh}$ is used. Hence, for the first two terms $J_{1},\ J_{2}$ of $\eqref{estimatedenoest}$, combining $\eqref{K0HQ1HQ2Hfirst}$ and $\eqref{K0HQ1HQ2Hthird}$, as long as sufficiently small $H$, we have
\begin{align}
J_{1}&=|b((B_{0}^{k})^{-1}Q_{1}^{H}\bm{u}^{k},Q_{1}^{H}\bm{u}^{k})|\notag
\\&=\frac{1}{|\lambda^{k}-\lambda_{1}^{H}|}||Q_{1}^{H}\bm{u}^{k}||_{b}^{2}
\geq \frac{C}{|\lambda^{k}-\lambda_{1}^{H}|},\label{denominatorfirst}
\end{align}
and
\begin{align}
J_{2}&\leq |b((B_{0}^{k})^{-1}Q_{2}^{H}\bm{u}^{k},Q_{2}^{H}\bm{u}^{k})|\notag\\
&\leq \frac{1}{\lambda_{2}^{H}-\lambda^{k}}||Q_{2}^{H}\bm{u}^{k}||^{2}_{b}
\leq \frac{CH^{2}}{\lambda_{2}^{H}-\lambda^{k}}
\leq CH^{2}.\label{denominatorsecond}
\end{align}
For the third term $J_{3}$ of $\eqref{estimatedenoest}$, by $\eqref{Bik}$ and Assumption $1$, we know
\begin{align}
J_{3}
&\leq CH^{2}\sum_{i=1}^{N}b(Z^{(i)}Q^{(i)}\bm{u}^{k},Z^{(i)}Q^{(i)}\bm{u}^{k})\notag\\
&\leq CH^{2}\sum_{i=1}^{N}||\bm{u}^{k}||_{b,\Omega_{i}^{'}}^{2}
\leq CH^{2}.\label{denominatorthird}
\end{align}
Combining $\eqref{denominatorfirst}\sim \eqref{denominatorthird}$ together, we get
\begin{equation}\notag
 |b((B_{h}^{k})^{-1}\bm{u}^{k},\bm{u}^{k})|\geq \frac{C}{|\lambda^{k}-\lambda_{1}^{H}|},
\end{equation}
which, together with $\eqref{estimatnum}$, proves this lemma. \qed
\end{proof}

\par Next, we estimate the almost counterbalanced term $I_{2}$. For convenience of analysis, we set $\hat{\bm{v}}_{1}^{H}:=Q_{1}^{H}Z^{(0)}Q^{H}(A^{h}-\lambda^{k})\bm{u}^{k},\ \check{\bm{v}}_{1}^{H}:=-\beta^{k}Q_{1}^{H}Z^{(0)}Q^{H}\bm{u}^{k}.$
 It is easy to see that \[\alpha Q_{2}^{h}(B_{0}^{k})^{-1}Q_{1}^{H}Z^{(0)}Q^{H}(A_{h}-\lambda^{k})\bm{u}^{k}
-\alpha\beta^{k}Q_{2}^{h}(B_{0}^{k})^{-1}Q_{1}^{H}Z^{(0)}Q^{H}\bm{u}^{k}=\alpha Q_{2}^{h}(B_{0}^{k})^{-1}(\hat{\bm{v}}_{1}^{H}+\check{\bm{v}}_{1}^{H}).\]
 For $s=1,\ 2$, we have
\begin{align}
Q_{s}^{h}\bm{t}^{k+1}&=Q_{s}^{h}\bm{e}^{k+1}=Q_{s}^{h}(B_{h}^{k})^{-1}\bm{r}^{k}+\beta^{k}Q_{s}^{h}(B_{h}^{k})^{-1}\bm{u}^{k}\notag\\
&=Q_{s}^{h}(B_{0}^{k})^{-1}Q_{1}^{H}Z^{(0)}Q^{H}\bm{r}^{k}+Q_{s}^{h}(B_{0}^{k})^{-1}Q_{2}^{H}Z^{(0)}Q^{H}\bm{r}^{k}+
Q_{s}^{h}\sum_{i=1}^{N}(B_{i}^{k})^{-1}Z^{(i)}Q^{(i)}\bm{r}^{k}\notag \\
&\ \ \ +\beta^{k}Q_{s}^{h}(B_{0}^{k})^{-1}Q_{1}^{H}Z^{(0)}Q^{H}\bm{u}^{k}+\beta^{k}Q_{s}^{h}(B_{0}^{k})^{-1}Q_{2}^{H}Z^{(0)}Q^{H}\bm{u}^{k}\notag\\
&\ \ \ +
\beta^{k}Q_{s}^{h}\sum_{i=1}^{N}(B_{i}^{k})^{-1}Z^{(i)}Q^{(i)}\bm{u}^{k}.\label{Qs}
\end{align}
Denote
\begin{align}
f(Q_{s}^{h})&:=||Q_{s}^{h}(B_{0}^{k})^{-1}Q_{2}^{H}Z^{(0)}Q^{H}\bm{r}^{k}||_{b}+
||Q_{s}^{h}\sum_{i=1}^{N}(B_{i}^{k})^{-1}Z^{(i)}Q^{(i)}\bm{r}^{k}||_{b}
\notag \\&\ \ \ \ +|\beta^{k}|||Q_{s}^{h}(B_{0}^{k})^{-1}Q_{2}^{H}Z^{(0)}Q^{H}\bm{u}^{k}||_{b}+
|\beta^{k}|||Q_{s}^{h}\sum_{i=1}^{N}(B_{i}^{k})^{-1}Z^{(i)}Q^{(i)}\bm{u}^{k}||_{b}.\ \ \ \ \ (s=1,2)\label{fQsh}
\end{align}
 Using Lemma $\ref{normequivalence}$, the Pythagorean theorem on $E_{h}(\Omega_{h})$ in the sense of $||\cdot||_{b}$ and Lemma $\ref{LemmaK0hQ2H}$, we have
\begin{align}
&\ \ \ \ ||Q_{2}^{h}(B_{0}^{k})^{-1}(\hat{\bm{v}}_{1}^{H}+\check{\bm{v}}_{1}^{H})||^{2}_{E^{k}}\leq C||Q_{2}^{h}(B_{0}^{k})^{-1}(\hat{\bm{v}}_{1}^{H}+\check{\bm{v}}_{1}^{H})||^{2}_{E^{h}}\notag\\
&=C\lambda_{1}^{h}||K_{0}^{h}(B_{0}^{k})^{-1}(\hat{\bm{v}}_{1}^{H}+\check{\bm{v}}_{1}^{H})||^{2}_{b}-
(\lambda_{1}^{h}-\lambda_{1}^{H})||(B_{0}^{k})^{-1}(\hat{\bm{v}}_{1}^{H}+\check{\bm{v}}_{1}^{H})||^{2}_{b}\notag\\
&\leq CH^{2}||(B_{0}^{k})^{-1}(\hat{\bm{v}}_{1}^{H}+\check{\bm{v}}_{1}^{H})||^{2}_{b}\label{combine1}.
\end{align}
Moreover,
\begin{align*}
&\ \ \ \ ||(B_{0}^{k})^{-1}(\hat{\bm{v}}_{1}^{H}+\check{\bm{v}}_{1}^{H})||^{2}_{b}\\
&=||K_{0}^{h}(B_{0}^{k})^{-1}(\hat{\bm{v}}_{1}^{H}+\check{\bm{v}}_{1}^{H})||^{2}_{b}+
||Q_{1}^{h}(B_{0}^{k})^{-1}(\hat{\bm{v}}_{1}^{H}+\check{\bm{v}}_{1}^{H})||^{2}_{b}+
||Q_{2}^{h}(B_{0}^{k})^{-1}(\hat{\bm{v}}_{1}^{H}+\check{\bm{v}}_{1}^{H})||^{2}_{b}\\
&\leq CH^{2}||(B_{0}^{k})^{-1}(\hat{\bm{v}}_{1}^{H}+\check{\bm{v}}_{1}^{H})||^{2}_{b}+
||Q_{1}^{h}(B_{0}^{k})^{-1}(\hat{\bm{v}}_{1}^{H}+\check{\bm{v}}_{1}^{H})||^{2}_{b}+
CH^{2}||(B_{0}^{k})^{-1}(\hat{\bm{v}}_{1}^{H}+\check{\bm{v}}_{1}^{H})||^{2}_{b}.
\end{align*}
Then
\[||(B_{0}^{k})^{-1}(\hat{\bm{v}}_{1}^{H}+\check{\bm{v}}_{1}^{H})||^{2}_{b}\leq C||Q_{1}^{h}(B_{0}^{k})^{-1}(\hat{\bm{v}}_{1}^{H}+\check{\bm{v}}_{1}^{H})||^{2}_{b}.\]
By $\eqref{combine1}$, we may obtain
\[ ||Q_{2}^{h}(B_{0}^{k})^{-1}(\hat{\bm{v}}_{1}^{H}+\check{\bm{v}}_{1}^{H})||_{E^{k}}\leq CH||Q_{1}^{h}(B_{0}^{k})^{-1}(\hat{\bm{v}}_{1}^{H}+\check{\bm{v}}_{1}^{H})||_{b}.\]
Furthermore, combining $\eqref{Qs}$ and $\eqref{fQsh}$, we have
\begin{align}
&\ \ \ \ ||Q_{2}^{h}(B_{0}^{k})^{-1}(\hat{\bm{v}}_{1}^{H}+\check{\bm{v}}_{1}^{H})||_{E^{k}}\leq CH||Q_{1}^{h}(B_{0}^{k})^{-1}(\hat{\bm{v}}_{1}^{H}+\check{\bm{v}}_{1}^{H})||_{b}\notag\\
&\leq CH||Q_{1}^{h}\bm{e}^{k+1}-Q_{1}^{h}(B_{0}^{k})^{-1}Q_{2}^{H}Z^{(0)}Q^{H}\bm{r}^{k}-Q_{1}^{h}\sum_{i=1}^{N}(B_{i}^{k})^{-1}Z^{(i)}Q^{(i)}\bm{r}^{k}\notag\\
&\ \ \ -\beta^{k}Q_{1}^{h}(B_{0}^{k})^{-1}Q_{2}^{H}Z^{(0)}Q^{H}\bm{u}^{k}-\beta^{k}Q_{1}^{h}\sum_{i=1}^{N}(B_{i}^{k})^{-1}Z^{(i)}Q^{(i)}\bm{u}^{k}||_{b}\notag\\
&\leq CH\{||Q_{1}^{h}\bm{e}^{k+1}||_{b}+f(Q_{1}^{h})\}= CH\{||Q_{1}^{h}\bm{t}^{k+1}||_{b}+f(Q_{1}^{h})\}.\label{Q1htk1}
\end{align}
 Using the orthogonal property of Jacobi-Davidson method and Helmholtz orthogonal decomposotion, we get
\begin{align*}
||Q_{1}^{h}\bm{t}^{k+1}||_{b}
&=||Q_{1}^{h}\bm{e}^{k+1}||_{b}=|b(\bm{u}_{1}^{h},\bm{e}^{k+1})|=|b(\bm{u}_{1}^{h}-\bm{u}^{k},\bm{e}^{k+1})|\\
&=|b(\bm{u}_{1}^{h}-\bm{u}^{k},\bm{e}^{k+1}-\nabla{p}_{h}^{k})|\leq ||\bm{u}_{1}^{h}-\bm{u}^{k}||_{b}||\bm{t}^{k+1}||_{b}.
\end{align*}
By $\eqref{u1e2definition}$, $\eqref{e2Eku1b}$ and the Poincar$\acute{e}$ inequality, for sufficiently small $H$, we have
\begin{align}
&\ \ \ \ ||\bm{u}_{1}^{h}-\bm{u}^{k}||_{b}\leq ||\bm{u}_{1}^{h}-\bm{u}_{1}^{k}+\bm{e}_{2}^{k}||_{b}\leq||\bm{u}_{1}^{h}-\bm{u}_{1}^{k}||_{b}+||\bm{e}_{2}^{k}||_{b}\notag\\
&= ||\bm{u}_{1}^{h}||_{b}-||\bm{u}_{1}^{k}||_{b}+||\bm{e}_{2}^{k}||_{b}
=||\bm{u}^{k}||_{b}-||\bm{u}_{1}^{k}||_{b}+||\bm{e}_{2}^{k}||_{b}\notag\\
&\leq 2||\bm{e}_{2}^{k}||_{b}\leq C||\bm{e}_{2}^{k}||_{E^{k}}\leq C\sqrt{\lambda^{k}-\lambda_{1}^{h}}. \label{Q1htk2}
\end{align}
Then
\[ ||Q_{1}^{h}\bm{t}^{k+1}||_{b}\leq C\sqrt{\lambda^{k}-\lambda_{1}^{h}}||\bm{t}^{k+1}||_{b}. \]
It is known that \[||\bm{t}^{k+1}||_{b}\leq ||Q_{1}^{h}\bm{t}^{k+1}||_{b}+||Q_{2}^{h}\bm{t}^{k+1}||_{b},\] we have
\[ ||\bm{t}^{k+1}||_{b}\leq C||Q_{2}^{h}\bm{t}^{k+1}||_{b}.\] Furthermore,
\begin{equation}\label{Q1ht}
 ||Q_{1}^{h}\bm{t}^{k+1}||_{b}\leq C\sqrt{\lambda^{k}-\lambda_{1}^{h}}||Q_{2}^{h}\bm{t}^{k+1}||_{b}.
\end{equation}
Combining $\eqref{Qs}$, $\eqref{fQsh}$, $\eqref{Q1htk1}$ and $\eqref{Q1ht}$, we obtain
\begin{align*}
&\ \ \ \ ||Q_{2}^{h}(B_{0}^{k})^{-1}(\hat{\bm{v}}_{1}^{H}+\check{\bm{v}}_{1}^{H})||_{E^{k}}\leq CH\{||Q_{1}^{h}\bm{t}^{k+1}||_{b}+f(Q_{1}^{h})\}\\
&\leq CHf(Q_{1}^{h})+CH\sqrt{\lambda^{k}-\lambda_{1}^{h}}\{||Q_{2}^{h}(B_{0}^{k})^{-1}(\hat{\bm{v}}_{1}^{H}+\check{\bm{v}}_{1}^{k})||_{b}+f(Q_{2}^{h})   \}\\
&\leq CHf(Q_{1}^{h})+CH\sqrt{\lambda^{k}-\lambda_{1}^{h}}\{||Q_{2}^{h}(B_{0}^{k})^{-1}(\hat{\bm{v}}_{1}^{H}+\check{\bm{v}}_{1}^{k})||_{E^{k}}+f(Q_{2}^{h})   \}.
\end{align*}
Then
\begin{equation}\label{Twocom}
 ||Q_{2}^{h}(B_{0}^{k})^{-1}(\hat{\bm{v}}_{1}^{H}+\check{\bm{v}}_{1}^{H})||_{E^{k}}\leq CHf(Q_{1}^{h})+CH\sqrt{\lambda^{k}-\lambda_{1}^{h}}f(Q_{2}^{h}).
\end{equation}
Next, we estimate the terms $f(Q_{1}^{h})$ and $f(Q_{2}^{h})$ in $\eqref{fQsh}$.
For the first term in $\eqref{fQsh}$, by the Poincar$\acute{e}$ inequality, $\eqref{e2Eku1b}$ and $\eqref{Gkcoarse}$, we get
\begin{align}
&\ \ \ \ ||(B_{0}^{k})^{-1}Q_{2}^{H}Z^{(0)}Q^{H}\bm{r}_{k}||_{b}\notag\\
&\leq ||(B_{0}^{k})^{-1}Q_{2}^{H}Z^{(0)}Q^{H}(A_{h}-\lambda^{k})\bm{e}_{2}^{k}||_{b}+
(\lambda^{k}-\lambda_{1}^{h})||(B_{0}^{k})^{-1}Q_{2}^{H}Z^{(0)}Q^{H}\bm{u}_{1}^{k}||_{b}\notag\\
&\leq ||T_{0}^{k}\bm{e}_{2}^{k}||_{b}+
C(\lambda^{k}-\lambda_{1}^{h})||Q_{2}^{H}\bm{u}_{1}^{k}||_{b}\notag\\
&\leq ||T_{0}^{k}\bm{e}_{2}^{k}||_{E^{k}}+
CH(\lambda^{k}-\lambda_{1}^{h})||\bm{u}_{1}^{k}||_{b}\leq C||\bm{e}_{2}^{k}||_{E_{k}}.\label{twocomestimate1}
\end{align}
For the second term in $\eqref{fQsh}$, by Lemma $\ref{strengthenedCauchySchwarzinequality}$, the Poincar$\acute{e}$ inequality, $\eqref{e2Eku1b}$ and $\eqref{Gkifinelocal}$, we have
\begin{align}
&\ \ \ \ ||\sum_{i=1}^{N}(B_{i}^{k})^{-1}Z^{(i)}Q^{(i)}\bm{r}^{k}||_{b}\notag\\
&\leq ||\sum_{i=1}^{N}(B_{i}^{k})^{-1}Z^{(i)}Q^{(i)}(A^{h}-\lambda^{k})\bm{e}_{2}^{k}||_{b}+
(\lambda^{k}-\lambda_{1}^{h})||\sum_{i=1}^{N}(B_{i}^{k})^{-1}Z^{(i)}Q^{(i)}\bm{u}_{1}^{k}||_{b}\notag\\
&\leq ||\sum_{i=1}^{N}T_{i}^{k}\bm{e}_{2}^{k}||_{b}+
CH^{2}(\lambda^{k}-\lambda_{1}^{h})||\sum_{i=1}^{N}Z^{(i)}Q^{(i)}\bm{u}_{1}^{k}||_{b}\leq CH^{2}||\bm{e}_{2}^{k}||_{E^{k}}.\label{twocomestimate2}
\end{align}
For the third term in $\eqref{fQsh}$, by $\eqref{u1e2definition}$, Lemma $\ref{LemmaK0hQ2H}$ and Lemma $\ref{lemmaorthogonalparameter}$, we obtain
\begin{align}
&\ \ \ \ |\beta^{k}|\ ||(B_{0}^{k})^{-1}Q_{2}^{H}Z^{(0)}Q^{H}\bm{u}^{k}||_{b}\notag\\
&\leq CH\sqrt{\lambda^{k}-\lambda_{1}^{h}}\{||(B_{0}^{k})^{-1}Q_{2}^{H}Z^{(0)}Q^{H}\bm{e}_{2}^{k}||_{b}+
||(B_{0}^{k})^{-1}Q_{2}^{H}Z^{(0)}Q^{H}\bm{u}_{1}^{k}||_{b}\}\notag\\
&\leq CH\sqrt{\lambda^{k}-\lambda_{1}^{h}}\{||\bm{e}_{2}^{k}||_{b}+
||Q_{2}^{H}\bm{u}_{1}^{k}||_{b}\}
\leq CH^{2}||\bm{e}_{2}^{k}||_{E^{k}}.\label{twocomestimate3}
\end{align}
For the fourth term in $\eqref{fQsh}$, by $\eqref{Bik}$ and Lemma $\ref{lemmaorthogonalparameter}$, we know
\begin{align}
&\ \ \ \ |\beta^{k}|\ ||\sum_{i=1}^{N}(B_{i}^{k})^{-1}Z^{(i)}Q^{(i)}\bm{u}^{k}||_{b}
\leq CH\sqrt{\lambda^{k}-\lambda_{1}^{h}}\cdot CH^{2}\sum_{i=1}^{N}||Z^{(i)}Q^{(i)}\bm{u}^{k}||_{b}\notag\\
&\leq CH^{3}\sqrt{\lambda^{k}-\lambda_{1}^{h}}||\bm{u}^{k}||_{b}\leq CH^{3}||\bm{e}_{2}^{k}||_{E^{k}}.\label{twocomestimate4}
\end{align}
Finally, we may finish the estimate of the almost counterbalanced term $I_{2}$ by combining Lemma $\ref{LemmaK0hQ2H}$, $\eqref{fQsh}$, $\eqref{Twocom}$ and $\eqref{twocomestimate1}\sim \eqref{twocomestimate4}$, i.e.,
\begin{equation}\label{Twocom1}
 ||Q_{2}^{h}(B_{0}^{k})^{-1}(\hat{\bm{v}}_{1}^{H}+\check{\bm{v}}_{1}^{k})||_{E^{k}}\leq CHf(Q_{1}^{h})+CH\sqrt{\lambda^{k}-\lambda_{1}^{h}}f(Q_{2}^{h})\leq CH^{2}||\bm{e}_{2}^{k}||_{E^{k}}.
\end{equation}
\subsection{The main result of this paper}
\par In this subsection, we first estimate the term $I_{3}$ in $\eqref{errorsplitting}$, and then we shall give our main result of this paper. For convenience, we denote $S^{k}:=Q_2^{h}(B^{k}_{0})^{-1}Q_{2}^{H}Z^{(0)}Q^{H}$. So \[Q_2^{h}(B^{k}_{0})^{-1}Q_{2}^{H}Z^{(0)}Q^{H}(A^{h}-\lambda^{k})\bm{u}_{1}^{k}=S^{k}(A^{h}-\lambda^{k})\bm{u}_{1}^{k}.\] For any $\bm{v}_{1}^{h}\in M_{h}(\lambda_{1})$, we have
\begin{align}
||S^{k}\bm{v}_{1}^{h}||^{2}_{E^{k}}
&=b((B^{k}_{0})^{-1}Q_{2}^{H}Z^{(0)}Q^{H}\bm{v}_{1}^{h},(A^{h}-\lambda^{k})S^{k}\bm{v}_{1}^{h})\notag\\
&=b(Q_{2}^{H}\bm{v}_{1}^{h},T_{0}^{k}S^{k}\bm{v}_{1}^{h})
\leq ||Q_{2}^{H}\bm{v}_{1}^{h}||_{b}||T_{0}^{k}S^{k}\bm{v}_{1}^{h}||_{b}\label{otherterms1}.
\end{align}
Combining Lemma $\ref{LemmaK0hQ2H}$, $\eqref{Gkcoarse}$ and $S^{k}\bm{v}_{1}^{h}\in M^{\perp}_{h}(\lambda_{1})$, we obtain
\begin{equation}\notag
||S^{k}\bm{v}_{1}^{h}||^{2}_{E^{k}}\leq CH||\bm{v}_{1}^{h}||_{b}||S^{k}\bm{v}_{1}^{h}||_{E^{k}}.
\end{equation}
Hence,
\begin{equation}\notag
||S^{k}\bm{v}_{1}^{h}||^{2}_{E^{k}}\leq CH^{2}||\bm{v}_{1}^{h}||_{b}^{2}.
\end{equation}
Moreover, we may obtain the following estimate
\begin{align*}
||S^{k}(A^{h}-\lambda^{k})\bm{u}_{1}^{k}||_{E^{k}}^{2}
=(\lambda^{k}-\lambda_{1}^{h})^2||S^{k}\bm{u}_{1}^{k}||_{E^{k}}^{2}
\leq CH^{2}(\lambda^{k}-\lambda_{1}^{h})^2||\bm{u}_{1}^{k}||_{b}^{2}\leq CH^{4}||\bm{e}_{2}^{k}||^{2}_{E^{k}},
\end{align*}
that is
\begin{equation}\label{othertermestimate1}
||Q_2^{h}(B^{k}_{0})^{-1}Q_{2}^{H}Z^{(0)}Q^{H}(A^{h}-\lambda^{k})\bm{u}_{1}^{k}||_{E^{k}}\leq CH^{2}||\bm{e}_{2}^{k}||_{E^{k}}.
\end{equation}

\par For the term $\beta^{k}Q_2^{h}(B^{k}_{0})^{-1}Q_{2}^{H}Z^{(0)}Q^{H}\bm{u}^{k}$ in $I_{3}$, by the similar argument with $\eqref{otherterms1}$, it is easy to see that
\[||Q_2^{h}(B^{k}_{0})^{-1}Q_{2}^{H}Z^{(0)}Q^{H}\bm{u}^{k}||^{2}_{E^{k}}\leq C||Q_{2}^{H}\bm{u}^{k}||^{2}_{b}.\] Hence,
 by Lemma $\ref{LemmaK0hQ2H}$, Lemma $\ref{lemmaorthogonalparameter}$ and $\eqref{e2Eku1b}$, we know
\begin{align}
&\ \ \ \ |\beta^{k}|||Q_2^{h}(B^{k}_{0})^{-1}Q_{2}^{H}Z^{(0)}Q^{H}\bm{u}^{k}||_{E^{k}}\notag\\
&\leq C|\beta^{k}|||Q_{2}^{H}\bm{u}^{k}||_{b}\leq C|\beta^{k}|\{||Q_{2}^{H}\bm{u}_{1}^{k}||_{b}+||Q_{2}^{H}\bm{e}_{2}^{k}||_{b}\}\notag\\
&\leq C|\beta^{k}|\{H||\bm{u}_{1}^{k}||_{b}+||\bm{e}_{2}^{k}||_{b}\}\leq CH\sqrt{\lambda^{k}-\lambda_{1}^{h}}(1+\frac{H}{\sqrt{\lambda^{k}-\lambda_{1}^{h}}})||\bm{e}_{2}^{k}||_{E^{k}}\notag\\
&\leq CH^{2}||\bm{e}_{2}^{k}||_{E^{k}}.\label{othertermestimate2}
\end{align}
\par For the terms $Q_2^{h}\sum_{i=1}^{N}(B_{i}^{k})^{-1}Z^{(i)}Q^{(i)}(A^{h}-\lambda^{k})\bm{u}_{1}^{k}$ and $\beta^{k}Q_2^{h}\sum_{i=1}^{N}(B_{i}^{k})^{-1}Z^{(i)}Q^{(i)}\bm{u}^{k}$ in $I_{3}$, we denote $F^{k}:=\sum_{i=1}^{N}(B_{i}^{k})^{-1}Z^{(i)}Q^{(i)}$. For any $\bm{v}^{h}\in E^{0}_{h}(\Omega_{h};\epsilon_{r})$, combining the Poincar$\acute{e}$ inequality and $\eqref{Gkifinelocal}$, we have
\begin{align*}
&\ \ \ \ ||Q_2^{h}F^{k}\bm{v}^{h}||^{2}_{E^{k}}
=\sum_{l=1}^{N}b((B_{l}^{k})^{-1}Z^{(l)}Q^{(l)}\bm{v}^{h},(A^{h}-\lambda^{k})Q_2^{h}F^{k}\bm{v}^{h})\\
&=\sum_{l=1}^{N}b(Q^{(l)}\bm{v}^{h},T_{l}^{k}Q_2^{h}F^{k}\bm{v}^{h})\leq(\sum_{l=1}^{N}||Q^{(l)}\bm{v}^{h}||^{2}_{b,\Omega_{l}^{'}})^{\frac{1}{2}}(\sum_{l=1}^{N}
||T_{l}^{k}Q_2^{h}F^{k}\bm{v}^{h}||^{2}_{b})^{\frac{1}{2}}\\
&\leq C\sqrt{N_{0}}H||\bm{v}^{h}||_{b}(\sum_{l=1}^{N}
||T_{l}^{k}Q_2^{h}F^{k}\bm{v}^{h}||^{2}_{E^{k}})^{\frac{1}{2}}\leq CH||\bm{v}^{h}||_{b}||Q_2^{h}F^{k}\bm{v}^{h}||_{E^{k}}.
\end{align*}
Then
\begin{equation}\label{otherterm3}
||Q_2^{h}\sum_{i=1}^{N}(B_{i}^{k})^{-1}Z^{(i)}Q^{(i)}\bm{v}^{h}||^{2}_{E^{k}}\leq CH^{2}||\bm{v}^{h}||^{2}_{b}\ \ \ \forall \ \bm{v}^{h}\in E^{0}_{h}(\Omega_{h};\epsilon_{r}).
\end{equation}
Specially, taking $\bm{v}^{h}=\bm{u}_{1}^{k}$, we obtain
\begin{align}
&\ \ \ \ ||Q_2^{h}\sum_{i=1}^{N}(B_{i}^{k})^{-1}Z^{(i)}Q^{(i)}(A^{h}-\lambda^{k})\bm{u}_{1}^{k}||^{2}_{E^{k}}\notag\\
&\leq(\lambda^{k}-\lambda_{1}^{h})^{2}||Q_2^{h}\sum_{i=1}^{N}(B_{i}^{k})^{-1}Z^{(i)}Q^{(i)}\bm{u}_{1}^{k}||^{2}_{E^{k}}\notag\\
&\leq CH^{2}(\lambda^{k}-\lambda_{1}^{h})^{2}||\bm{u}_{1}^{k}||^{2}_{b}\leq CH^{4}||\bm{e}_{2}^{k}||^{2}_{E^{k}}.\label{othertermestimate3}
\end{align}
 and similarly taking $\bm{v}^{h}=\bm{u}_{1}^{k},\bm{v}^{h}=\bm{e}_{2}^{k}$, we then have
\begin{align}
&\ \ \ \ |\beta^{k}|||Q_2^{h}\sum_{i=1}^{N}(B_{i}^{k})^{-1}Z^{(i)}Q^{(i)}\bm{u}^{k}||_{E^{k}}\notag\\
&\leq |\beta^{k}|\{||Q_2^{h}\sum_{i=1}^{N}(B_{i}^{k})^{-1}Z^{(i)}Q^{(i)}\bm{u}_{1}^{k}||_{E^{k}}+
||Q_2^{h}\sum_{i=1}^{N}(B_{i}^{k})^{-1}Z^{(i)}Q^{(i)}\bm{e}_{2}^{k}||_{E^{k}}\}\notag\\
&\leq |\beta^{k}|\{CH||\bm{u}_{1}^{k}||_{b}+CH||\bm{e}_{2}^{k}||_{b}\}
\leq CH^{2}\sqrt{\lambda^{k}-\lambda_{1}^{h}}(1+\frac{1}{\sqrt{\lambda^{k}-\lambda_{1}^{h}}})||\bm{e}_{2}^{k}||_{E^{k}}\notag\\
&\leq CH^{2}||\bm{e}_{2}^{k}||_{E^{k}}.\label{othertermestimate4}
\end{align}
 \par Combining Theorem $\ref{Gkv2theorem}$, Lemma $\ref{lemmaorthogonalparameter}$,  $\eqref{Twocom1}$, $\eqref{othertermestimate1}$, $\eqref{othertermestimate2}$, $\eqref{othertermestimate3}$ and $\eqref{othertermestimate4}$, we may obtain the main convergence result of this paper.

\begin{theorem}
For the Algorithm 2, the discrete principal eigenvalue of the Maxwell eigenvalue problem satisfies
\begin{equation}\label{conver1}
||\bm{e}_{2}^{k+1}||_{E}\leq c(H)(1-C\frac{\delta^{2}}{H^{2}})||\bm{e}_{2}^{k}||_{E},
\end{equation}
and
\begin{equation}\label{conver2}
\lambda^{k+1}-\lambda_{1}^{h}\leq c(H)(1-C\frac{\delta^{2}}{H^{2}})^{2}(\lambda^{k}-\lambda_{1}^{h})
\end{equation}
where $c(H)\to 1$, as $H\to 0$ and $C$ is independent of $h,\ H,\ \delta$.
\end{theorem}

\begin{proof}
By $\eqref{lowerboundestimate}$, $\eqref{upperboundestimate}$ and Lemma $\ref{normequivalence}$, we have
\begin{align*}
&\ \ \ \ ||\widetilde{\bm{e}}_{2}^{k+1}||_{E}^{2}=||\widetilde{\bm{e}}_{2}^{k+1}||_{E^{k}}^{2}+(\lambda^{k}-\lambda_{1})||\widetilde{\bm{e}}_{2}^{k+1}||_{b}^{2}\\
&\leq \{(1-C\frac{\delta^{2}}{H^{2}})^{2}+CH^{2}\}||\bm{e}_{2}^{k}||_{E^{k}}^{2}\leq c(H)(1-C\frac{\delta^{2}}{H^{2}})^{2}||\bm{e}_{2}^{k}||_{E}^{2},
\end{align*}
where $c(H)\to 1$, as $H\to 0$. Combining the orthogonal property of Jacobi-Davidson method and Helmholtz orthogonal decomposition, we have
$b(\bm{u}^{k},\bm{t}^{k+1})=0$. Due to $\eqref{specialcase1}$ and $\eqref{specialcase2}$, we have
\begin{equation}\notag
||\widetilde{\bm{u}}^{k+1}||_{b}^{2}=||\bm{u}^{k}||_{b}^{2}+\alpha^{2}||\bm{t}^{k+1}||_{b}^{2}\geq 1.
\end{equation}
Then
\begin{equation}\notag
||\check{\bm{e}}_{2}^{k+1}||^{2}_{E}\leq ||\widetilde{\bm{e}}_{2}^{k+1}||^{2}_{E}\leq c(H)(1-C\frac{\delta^{2}}{H^{2}})^{2}||\bm{e}_{2}^{k}||_{E}^{2}.
\end{equation}
By Lemma $\ref{normequivalence}$, we may obtain
\begin{equation}\notag
||\check{\bm{e}}_{2}^{k+1}||^{2}_{E^{h}}\leq c(H)(1-C\frac{\delta^{2}}{H^{2}})^{2}||\bm{e}_{2}^{k}||_{E^{h}}^{2}.
\end{equation}
Moreover,
\begin{equation}\notag
\check{\lambda}^{k+1}-\lambda_{1}^{h}\leq c(H)(1-C\frac{\delta^{2}}{H^{2}})^{2}(\lambda^{k}-\lambda_{1}^{h}).
\end{equation}
We know that $\check{\bm{u}}^{k+1}$ is a special vector in $W^{k+1}$ but is not stable function when minimizes the Rayleigh quotient. Hence,
\begin{equation}\notag
\lambda^{k+1}-\lambda_{1}^{h}\leq \check{\lambda}^{k+1}-\lambda_{1}^{h} \leq c(H)(1-C\frac{\delta^{2}}{H^{2}})^{2}(\lambda^{k}-\lambda_{1}^{h}).
\end{equation}
Then
\begin{equation}\notag
||\bm{e}_{2}^{k+1}||_{E^{h}}\leq c(H)(1-C\frac{\delta^{2}}{H^{2}})||\bm{e}_{2}^{k}||_{E^{h}}.
\end{equation}
Moreover,
\begin{equation}\notag
||\bm{e}_{2}^{k+1}||_{E}\leq c(H)(1-C\frac{\delta^{2}}{H^{2}})||\bm{e}_{2}^{k}||_{E},
\end{equation}
which completes the proof of the theorem.\qed
\end{proof}

\section{Numerical experiments}
\par In this section, we shall present several numerical experiments in two and three dimensional cases to support our theoretical findings. We do the computation in double decision but only display four digits after the decimal in tables except $\lambda^{it.}$ in order to show the convergent process. In 2D cases, the stopping tolerance is endowed by $|\lambda^{k+1}-\lambda^{k}|<10^{-8}$. In 3D cases, the stopping tolerance is endowed by $||\bm{r}^{k}||_{b}<10^{-5}$.
\subsection{2D Maxwell eigenvalue problems}
\par In this subsection, we shall present some numerical results for the 2D Maxwell eigenvalue problem in rectangle domain, square domain and L-shaped domain.
 \begin{example}
  We consider the Maxwell eigenvalue problem $\eqref{MaxwellEigenvalueu}$ on the two dimensional domain $\Omega=(0,2\pi)\times (0,\pi)$ with $\epsilon_{r}=\mu_{r}=1$ and use the lowest triangle edge element to compute the principal eigenvalue which is $\lambda_{1}=0.25$ with algebraic multiplicity $m_{1}=1$. First, we choose an initial uniform partition $\tau_{H}$ in $\Omega$ with the number of subdomains $N=64$, coarse grid size $H=\frac{\sqrt{2}\pi}{2^2}$. we refine uniformly the grid layer by layer and fix the ratio $\frac{\delta}{H}=\frac{1}{4}$. Next, we test the optimality and scalability of our PHJD algorithm.
  \end{example}
\begin{table}[H]
 \centering
 \caption{The number of subdomains $N=64$, the ratio $\frac{\delta}{H}=\frac{1}{4}$}
\newcolumntype{d}{D{.}{.}{2}}
\begin{tabular}{|c|c|c|c|c|c|c|}
\hline
$h$&\multicolumn{1}{c|}{$d.o.f.$}&\multicolumn{1}{c|}{$it.$}&
\multicolumn{1}{c|}{$|\lambda^{k+1}-\lambda^{k}|$}&\multicolumn{1}{c|}{$||\bm{r}^{k}||_{b}$}&\multicolumn{1}{c|}{$\lambda^{it.}$}
&\multicolumn{1}{c|}{$Con.ord.$}\\
\hline
$\frac{\sqrt{2}\pi}{2^4}$&1488&8&8.8778e-10&1.1214e-04&0.249933057840449&-\\
\hline
$\frac{\sqrt{2}\pi}{2^5}$&6048&7&1.1871e-09&2.5439e-04&0.249983266630354&2.0002\\
\hline
$\frac{\sqrt{2}\pi}{2^{6}}$&24384&7&1.2434e-09&2.1708e-04&0.249995816760741&2.0000\\
\hline
$\frac{\sqrt{2}\pi}{2^{7}}$&97920&7&7.5604e-10&1.3077e-04&0.249998954192258&2.0000\\
\hline
$\frac{\sqrt{2}\pi}{2^{8}}$&392448&7&6.9670e-10&1.0839e-04&0.249999738755104&2.0011\\
\hline
$\frac{\sqrt{2}\pi}{2^{9}}$&1571328&7&1.3278e-10&8.8096e-04&0.249999934697985&2.0002\\
\hline
$\frac{\sqrt{2}\pi}{2^{10}}$&6288384&7 &1.4611e-10&9.4527e-04&0.249999983669919&1.9996\\
\hline
\end{tabular}
\end{table}

\begin{table}[H]
 \centering
 \caption{The number of subdomains $N=64, 256, 1024$}
\newcolumntype{d}{D{.}{.}{2}}
\begin{tabular}{|c|c|c|}
\hline
$N$&\multicolumn{1}{c|}{$d.o.f$}&\multicolumn{1}{c|}{$it.$}\\
\hline
64&6288384&7\\
\hline
256&6288384&6\\
\hline
1024&6288384&6\\
\hline
\end{tabular}
\end{table}
\par The notation $h$ means the fine grid size. $d.o.f.$ means the degree of freedom. $it.$ means iterative steps in our algorithm. $Con.ord.$ means the convergence order of the principal eigenvalue. It is shown in Table 1 that the convergence rate does not deteriorate when $h\to 0$. Meanwhile, the PHJD algorithm works very well when $h<<H^{4}$ which coincides with our theory. In order to verify the scalability of our PHJD method, we choose the number of domains $N=64, 256, 1024$ to count the iterative steps for $d.o.f=6288384 $.
The iterative steps decreasing in table 2 illustrates that our method based on domain decomposition is scalable, which coincides with our theory. Actually, the near optimality and scalability of our method hold not only for simple case but also for multiple case.
\begin{example}
 We consider the Maxwell eigenvalue problem $\eqref{MaxwellEigenvalueu}$ in $(0,\pi)^{2}$ with $\epsilon_{r}=1, \mu_{r}=1$ and use the lowest triangle edge element to compute the principal eigenvalue which is $\lambda_{1}=1$ with algebraic multiplicity $m_{1}=2$. First, we choose an initial uniform partition $\tau_{H}$ in $\Omega$ with the number of subdomains $N=8$, coarse grid size $H=\frac{\sqrt{2}\pi}{2}$. we refine uniformly the grid layer by layer and fix the ratio $\frac{\delta}{H}=\frac{1}{8}$. Next, we also test the optimality and scalability of our PHJD algorithm.
\end{example}
\begin{table}[H]
 \centering
 \caption{The number of subdomains $N=8$, the ratio $\frac{\delta}{H}=\frac{1}{8}$}
\newcolumntype{d}{D{.}{.}{2}}
\begin{tabular}{|c|c|c|c|c|c|c|}
\hline
$h$&\multicolumn{1}{c|}{$d.o.f.$}&\multicolumn{1}{c|}{$it.$}&
\multicolumn{1}{c|}{$|\lambda^{k+1}-\lambda^{k}|$}&\multicolumn{1}{c|}{$||\bm{r}^{k}||_{b}$}&\multicolumn{1}{c|}{$\lambda^{it.}$}
&\multicolumn{1}{c|}{$Con.ord.$}\\
\hline
$\frac{\sqrt{2}\pi}{2^4}$&736&7&2.1972e-09& 2.0136e-04&0.998065902158390&-\\
\hline
$\frac{\sqrt{2}\pi}{2^5}$&3008&7&3.2392e-09&1.8341e-04&0.999515561847595&1.9973\\
\hline
$\frac{\sqrt{2}\pi}{2^6}$&12160&7& 1.0886e-09&3.0479e-04&0.999878833291020& 1.9993\\
\hline
$\frac{\sqrt{2}\pi}{2^{7}}$&48896&7&1.5651e-09&1.8827e-04&0.999969704716447& 1.9998 \\
\hline
$\frac{\sqrt{2}\pi}{2^{8}}$&196096&7& 1.1096e-09&8.5103e-05&0.999992425964469&2.0000\\
\hline
$\frac{\sqrt{2}\pi}{2^{9}}$&785408&7&1.6332e-09&2.1306e-04&0.999998106489943& 2.0000\\
\hline
$\frac{\sqrt{2}\pi}{2^{10}}$&3143680&7&1.2116e-10&6.8551e-04&0.999999526630577&  2.0000\\
\hline
$\frac{\sqrt{2}\pi}{2^{11}}$&12578816&7&1.0379e-10&9.3310e-04&0.999999881662115&2.0001\\
\hline
\end{tabular}
\end{table}
\begin{table}[H]
 \centering
 \caption{The number of subdomains $N=8,32,128,512$}
\newcolumntype{d}{D{.}{.}{2}}
\begin{tabular}{|c|c|c|}
\hline
$N$&\multicolumn{1}{c|}{$d.o.f.$}&\multicolumn{1}{c|}{$it.$}\\
\hline
8&12578816&7\\
\hline
32&12578816&6\\
\hline
128&12578816&6\\
\hline
512&12578816&6\\
\hline
\end{tabular}
\end{table}
\par Although we only give the theoretical analysis for the first simple eigenvalue case, our algorithm still works well and keeps good scalability for the first multiple eigenvalue case. It is shown in Table 3 that the convergence rate does not deteriorate when $h\to 0$. Meanwhile, our PHJD algorithm works very well when $h<<H^{4}$. Similarly, we choose the number of domains $N=8, 32, 128, 512$ to count the iterative steps for $d.o.f.=12578816$. Table 4 shows that the iterative steps keep stable when the number of domains $N$ increases, which illustrates that our PHJD method is scalable for the multiple principal eigenvalue case.
\begin{example}
We consider the Maxwell eigenvalue problem in  $(-1,1)^{2}\backslash[0,1)\times(-1,0]$ with $\epsilon_{r}=\mu_{r}=1$ and use the lowest triangle edge element to compute the principal eigenvalue which is $1.47562182$ with algebraic multiplicity $m_{1}=1$ (see \cite{Buffa2007Discontinuous,chenlong}). First, we choose an initial uniform partition $\tau_{H}$ in $\Omega$ with the number of subdomains $N=96$, coarse grid size $H=\frac{\sqrt{2}}{2^{2}}$. We refine uniformly the grid layer by layer and fix the ratio $\frac{\delta}{H}=\frac{1}{8}$. Next, we test the optimality and scalability of our PHJD algorithm.
\begin{table}[H]
 \centering
 \caption{The number of subdomains $N=96$, the ratio $\frac{\delta}{H}=\frac{1}{8}$}
\newcolumntype{d}{D{.}{.}{2}}
\begin{tabular}{|c|c|c|c|c|c|}
\hline
$h$&\multicolumn{1}{c|}{$d.o.f.$}&\multicolumn{1}{c|}{$it.$}&
\multicolumn{1}{c|}{$|\lambda^{k+1}-\lambda^{k}|$}&\multicolumn{1}{c|}{$\lambda^{it.}$}
&\multicolumn{1}{c|}{$Con.ord.$}\\
\hline
$\frac{\sqrt{2}}{2^5}$&9088&8&2.4067e-09&1.472164089752624 &-\\
\hline
$\frac{\sqrt{2}}{2^6}$&36608&7& 9.7419e-09&1.474258883109039&1.3431\\
\hline
$\frac{\sqrt{2}}{2^{7}}$&146944&7&2.8876e-09&1.475083316185965&1.3397 \\
\hline
$\frac{\sqrt{2}}{2^{8}}$&588800&7& 4.1055e-10 &1.475408720712429&1.3374\\
\hline
$\frac{\sqrt{2}}{2^{9}}$&2357248&7&8.3497e-11&1.475537406431712&1.3360\\
\hline
$\frac{\sqrt{2}}{2^{10}}$&9433088&7&2.8653e-11&1.475588361248540&1.3351\\
\hline
\end{tabular}
\end{table}
\begin{table}[H]
 \centering
 \caption{The number of subdomains $N=96,384,1536$}
\newcolumntype{d}{D{.}{.}{2}}
\begin{tabular}{|c|c|c|}
\hline
$N$&\multicolumn{1}{c|}{$d.o.f.$}&\multicolumn{1}{c|}{$it.$}\\
\hline
96&9433088&7\\
\hline
384&9433088&7\\
\hline
1536&9433088&7\\
\hline
\end{tabular}
\end{table}
\end{example}
\par We note that the first Maxwell eigenvector in L-shaped domain has a strong unbounded singularity at the re-entrant corner but our PHJD method still works very well. It is shown in Table 5 that the iterative steps $it.$ keep stable when $ h\to 0 $, which illustrates that the convergence rate of our algorithm is independent of the fine grid size $h$. The fact that iterative steps keep stable with the number of subdomains increasing in Table 6 illustrates our algorithm is scalable. Although we only give the theoretical analysis for convex domain, our PHJD method works well and may be extended to more general Lipschitz domain.
\subsection{3D Maxwell eigenvalue problems}
\par In this subsection, we shall present some numerical results for the 3D Maxwell eigenvalue problem in cuboid domain and cube domain.
\begin{example}
 We consider the Maxwell eigenvalue problem in  $(0,\pi)\times (0,2\pi)\times (0,1.2\pi)$ with $\epsilon_{r}=\mu_{r}=1$ and use the lowest cuboid edge element to compute the principal eigenvalue which is $\frac{17}{18}=0.9\dot{4}$ with algebraic multiplicity $m_{1}=1$. First, we choose an initial uniform partition $\tau_{H}$ in $\Omega$ with the number of subdomains $N=128$, coarse grid size $H=\frac{1.2\pi}{2^2}$. We refine uniformly the grid layer by layer and fix the ratio $\frac{\delta}{H}=\frac{1}{4}$. Next, we test the optimality and scalability of our PHJD algorithm.
 \end{example}
\begin{table}[H]
 \centering
 \caption{The number of subdomains $N=128$, the ratio $\frac{\delta}{H}=\frac{1}{4}$}
\newcolumntype{d}{D{.}{.}{2}}
\begin{tabular}{|c|c|c|c|c|c|}
\hline
$h$&\multicolumn{1}{c|}{$d.o.f$}&\multicolumn{1}{c|}{$it.$}&
\multicolumn{1}{c|}{$||\bm{r}^{k}||_{b}$}
&\multicolumn{1}{c|}{$\lambda^{it.}$}&\multicolumn{1}{c|}{$Con.ord.$}\\
\hline
$\frac{1.2\pi}{2^{4}}$&22080&9& 5.4339e-06&0.946879258942834&-\\
\hline
$\frac{1.2\pi}{2^{5}}$&186496&8&3.1060e-06&0.945052693133526& 2.0011\\
\hline
$\frac{1.2\pi}{2^{6}}$&1532160&8&7.6311e-06&0.944598398293954&1.9822\\
\hline
\end{tabular}
\end{table}
\begin{table}[H]
 \centering
 \caption{The number of subdomains $N=128, 1024$}
\newcolumntype{d}{D{.}{.}{2}}
\begin{tabular}{|c|c|c|}
\hline
$N$&\multicolumn{1}{c|}{$d.o.f.$}&\multicolumn{1}{c|}{$it.$}\\
\hline
128&1532160&8\\
\hline
1024&1532160&7\\
\hline
\end{tabular}
\end{table}
\par It is shown in Table 7 that the iterative steps keep stable when $ h\to 0 $, which illustrates that the convergence rate of our algorithm is independent of the fine grid size $h$. The fact that iterative steps decreases in Table 8 illustrates our algorithm is scalable. These coincide with our theory in this paper. Actually, the near optimality and scalability of our method hold not only for simple case but also for multiple case in 3D case.
\begin{example}
We consider the Maxwell eigenvalue problem in  $(0,\pi)^{3}$ with $\epsilon_{r}=\mu_{r}=1$ and use the lowest cuboid edge element to compute the principal eigenvalue which is $\lambda_{1}=2$ with algebraic multiplicity $m_{1}=3$. First, we choose an initial uniform partition $\tau_{H}$ in $\Omega$ with the number of subdomains $N=64$, coarse grid size $H=\frac{\pi}{2^2}$. We refine uniformly the grid layer by layer and fix the ratio $\frac{\delta}{H}=\frac{1}{4}$. Next, we also test the optimality and scalability of our PHJD algorithm.
\end{example}
\begin{table}[H]
 \centering
 \caption{The number of subdomains $N=64$, the ratio $\frac{\delta}{H}=\frac{1}{4}$}
\newcolumntype{d}{D{.}{.}{2}}
\begin{tabular}{|c|c|c|c|c|c|}
\hline
$h$&\multicolumn{1}{c|}{$d.o.f$}&\multicolumn{1}{c|}{$it.$}&
\multicolumn{1}{c|}{$||\bm{r}^{k}||_{b}$}
&\multicolumn{1}{c|}{$\lambda^{it.}$}&\multicolumn{1}{c|}{$Con.ord.$}\\
\hline
$\frac{\pi}{2^{4}}$&10800&9&7.6503e-06&2.006433745425857&-\\
\hline
$\frac{\pi}{2^{5}}$&92256&8&4.5546e-06&2.001607079135022&2.0012\\
\hline
$\frac{\pi}{2^{6}}$&762048&7&9.8823e-06&2.000419881720004&1.9364\\
\hline
$\frac{\pi}{2^{7}}$&6193536&7&4.1799e-06&2.000105199582128&1.9969\\
\hline
\end{tabular}
\end{table}
\begin{table}[H]
 \centering
 \caption{The number of subdomains $N=64, 512$}
\newcolumntype{d}{D{.}{.}{2}}
\begin{tabular}{|c|c|c|}
\hline
$N$&\multicolumn{1}{c|}{$d.o.f.$}&\multicolumn{1}{c|}{$it.$}\\
\hline
64&6193536&7\\
\hline
512&6193536&6\\
\hline
\end{tabular}
\end{table}
\par Although we only give the theoretical analysis for the first simple eigenvalue case, our algorithm still works well and keeps good scalability for the first multiple eigenvalue case. The iterative steps keep stable in Table $9$ when $h\to 0$, which illustrates that the convergence rate is independent of $h$. It is shown in Table 10 that the iterative steps decrease, which illustrates that the PHJD method is scalable.
\section{Conclusions}
\par In this paper, based on the domain decomposition method, we propose a new and robust two-level PHJD method for solving the Maxwell eigenvalue problem. The two-level PHJD method has a good scalability and is asymptotically optimal without any assumption between coarse size $H$ and fine size $h$. Numerical results confirm our theoretical findings.
\begin{small}
\bibliographystyle{plain}
\bibliography{reference}
\end{small}
\end{document}